\documentclass[a4paper, 11pt]{amsart}

\usepackage{amscd, amsfonts,  amssymb, amsbsy, latexsym}
\usepackage[all]{xy}
\usepackage{graphicx}
\usepackage{cite}
\usepackage[table,xcdraw]{xcolor}
\usepackage{multirow}
\usepackage{hyperref}
\advance\hoffset by -1in
\advance\voffset by -1in

\newtheorem{lemma}{Lemma}[section]
\newtheorem{theorem}[lemma]{Theorem}
\newtheorem{prop}[lemma]{Proposition}

\newtheorem{cor}[lemma]{Corollary}

\theoremstyle{definition}

\newtheorem{example}[lemma]{Example}
\newtheorem{remark}[lemma]{Remark}

\newenvironment{proof_of}[1]{\medskip\noindent{\it Proof #1}}
{\hfill$\Box$ \bigskip}

\makeatletter
\newcounter{bwt}

\makeatother

\DeclareMathOperator{\CC}{\mathbb{C}}
\DeclareMathOperator{\QQ}{\mathbb{Q}}
\DeclareMathOperator{\FF}{\mathbb{F}}
\DeclareMathOperator{\NN}{\mathbb{N}}

\DeclareMathOperator{\B}{\mathcal{B}}
\DeclareMathOperator{\Id}{\mathrm{Id}}

\DeclareMathOperator{\Sym}{\mathcal{S}}
\DeclareMathOperator{\A}{\mathcal{A}}

\DeclareMathOperator{\I}{\mathcal{I}}
\DeclareMathOperator{\J}{\mathcal{J}}
\DeclareMathOperator{\CCX}{\mathbb{C}\langle X\rangle}

\DeclareMathOperator{\dg}{\mathrm deg}
\DeclareMathOperator{\mdeg}{\mathrm mdeg}

\newenvironment{eq}{\begin{equation}}{\end{equation}}

\newcommand{\al}{\alpha}
\newcommand{\be}{\beta}
\newcommand{\ga}{\gamma}
\newcommand{\de}{\delta}
\newcommand{\si}{\sigma}

\newcommand{\lin}{\mathop{\rm lin}}

\newcommand{\ov}[1]{\overline{#1}}
\newcommand{\un}[1]{{\underline{#1}} }

\title[Novikov algebras in low dimension: identities, images and codimensions]%
{Novikov algebras in low dimension: identities, images and codimensions}

\author{Iritan Ferreira dos Santos, Alexey M. Kuz'min, and Artem Lopatin}

\address{Iritan Ferreira dos Santos, Universidade Estadual de Campinas (UNICAMP), 651 Sergio Buarque de Holanda, 13083-859 Campinas, SP, Brazil}
\email{ i195167@dac.unicamp.br (Iritan Ferreira dos Santos)}

\address{Alexey M. Kuz’min, Universidade Federal do Rio Grande do Norte, Departamento
de Mathemática, Centro de Ciências Exatas e da Terra, Campus Universitário,
Lagoa Nova, Natal, RN, 59078-970, Brazil}
\email{ amkuzmin@ya.ru (Alexey M. Kuz’min) }

\address{Artem Lopatin, Universidade Estadual de Campinas (UNICAMP), 651 Sergio Buarque de Holanda, 13083-859 Campinas, SP, Brazil}
\email{ dr.artem.lopatin@gmail.com (Artem Lopatin)}

\begin{document}

\maketitle

\begin{abstract}
Polynomial identities of two-dimensional Novikov algebras are studied over the complex field $\CC$. We determine minimal generating
sets for the T-ideals of the polynomial identities and linear
bases for the corresponding relatively free algebras. As a consequence, we establish  that polynomial identities separate two-dimensional Novikov algebras, which are not associative. Namely, any two-dimensional Novikov algebras, which are not associative, are isomorphic if and only if they satisfy the same polynomial identities. Moreover, we obtain the codimension sequences of all these algebras. In particular, every two-dimensional Novikov algebra has at most linear growth of its  codimension sequence.  We explicitly describe multilinear images of every two-dimensional Novikov algebra. In particular, we show that these images are vector spaces.


\end{abstract}

%\tableofcontents

%=================================================================
%=================================================================
\section{Introduction}

%=================================================================
\subsection{Novikov algebras} Throughout the paper, by an algebra we mean an $\FF$-vector space endowed with a bilinear product, where the field $\FF$ has an arbitrary characteristic $p\geqslant 0$. An algebra $\mathcal N$ is called a \textit{Novikov algebra} if $\mathcal N$ satisfies the 
following polynomial identities
\begin{align}
(x,y,z)&=(y,x,z)\qquad\textit{(the identity of left symmetry for associators)},
\label{left-symm}\\
(xy)z&=(xz)y\qquad\textit{(the identity of right commutativity)},
\label{right-commut}
\end{align}
where $(x,y,z)=(xy)z-x(yz)$ stands for the \textit{associator} in variables~$x,y,z$.  Novikov algebras were independently introduced in the study of Hamiltonian operators in the formal calculus of variations by Gel'fand and Dorfman~\cite{gel1979hamiltonian} as well as in the connection with linear Poisson brackets of hydrodynamic type by Balinskii and  Novikov~\cite{balinskii1985poisson}.

Novikov algebras are tightly connected with Lie algebras. Namely, for an algebra $\A$ denote by $\A^{(-)}$ the vector space $\A$ endowed with a new product $[-, -]$ defined by $[x, y] := x y-y x$ for all $x,y\in \A$. An algebra $\A$ is called a \textit{Lie-admissible algebra} if $\A^{(-)}$ is a Lie algebra. In this case, we call $\A^{(-)}$ the \textit{associated Lie algebra} of $\A$. 
%Recall that any algebra satisfying polynomial identity~\eqref{left-symm} is called \textit{left-symmetric algebra}. 
Every algebra satisfying polynomial identity~\eqref{left-symm} (and hence every Novikov algebra) is Lie-admissible.

Zelmanov~\cite{zel1987class} proved that every simple finite-dimensional Novikov algebra over an algebraically closed field of characteristic zero is one-dimensional.  Filippov~\cite{filippov1989class} constructed first examples of simple infinite-dimensional Novikov algebras over fields of characteristic $p\geqslant 0$ and simple finite-dimensional Novikov algebras over fields of characteristic $p>0$. Xu~\cite{Xu_1996} gave a  complete classification of finite-dimensional simple Novikov algebras over an algebraically closed field of characteristic $p>2$. Xu~\cite{Xu_2001} also established a classification of simple infinite-dimensional Novikov algebras over an algebraically closed field of characteristic zero.  

Bai and Meng~\cite{bai2001classification} classified all Novikov algebras of dimension two and three over $\CC$. There exists, up to isomorphism, seven two-dimensional Novikov algebras with nontrivial multiplication and a one-paremeter family $\mathbf{N}^{\ell}_6$ of two-dimensional Novikov algebras with nontrivial multiplication, where $l\in\CC\backslash\{0,1\}$  (see Theorem~\ref{theor-bai-meng} for more details). Moreover, Novikov algebras over $\CC$ were also classified in higher dimensions. Namely,  four-dimensional Novikov algebras with nilpotent associated Lie algebra were classified in~\cite{Burde_Graaf_2013} and  four-dimensional nilpotent Novikov algebras were described in~\cite{Kay_Novikov2019}. One-generated five- and six-dimesional nilpotent Novikov algebras were classified in~\cite{Kay_Novikov2022}. The geometric classification of three-dimensional Novikov algebras was given in~\cite{Benes_Burde_2014} and of four-dimensional nilpotent Novikov algebras in~\cite{Kay_Novikov2019}. Five-dimensional nilpotent Novikov algebras were classified in~\cite{Kay_Novikov2023}.

%=================================================================
\subsection{Polynomial identities} Although the theory of algebras with polynomial identities is a well-developed area of algebra with many deep results, there are still few results with explicit description of generators of the T-ideals of polynomial identities for particular finite-dimensional algebras. As an example, polynomial identities for $2\times 2$ matrices are known over a field of characteristic different from two (see~\cite{Drensky_1981_2x2, Koshlukov_2001_2x2, Colombo_Koshlukov_2004_2x2}), but the case of infinite field with characteristic two is still open. Moreover, the generators for the T-ideal of $3\times3$ matrices are not known over any infinite field.  Another open problem is the description of polynomial identities for the three-dimensional simple Lie algebra $\FF^3$ equipped with the cross product over a field of characteristic two.  Our main result is the description of polynomial identities for every two-dimensional Novikov algebra over the complex numbers $\CC$.

The classification of all two-dimensional algebras over an arbitrary algebraically closed field of any characteristic can be found in~\cite{kaygorodov2019variety} (see also~\cite{Ananin_Mironov_2000, Petersson_2000, Goze_Remm_2011}). A classification of two-dimensional algebras with respect to ``classical''{} polynomial identities was given by Ahmed, Bekbaev, Rakhimov~\cite{Ahmed_Bekbaev_Rakhimov_2020} over a field, where any second and third degree polynomial has a root.  Drensky~\cite{drensky2019varieties} applied the classification from~\cite{kaygorodov2019variety} to study two-dimensional bicommutative algebras and to describe its polynomial identities over an arbitrary field of characteristic zero. Moreover, he showed that one of these algebras generates the variety of all bicommutative algebras. 

Polynomial identities of an algebra are tightly connected with a basis of a corresponding relatively free algebra. Ferreira dos Santos, and Kuz'min~\cite{santos2023free} obtained an explicit linear basis and the multiplication table for the free weakly Novikov metabelian algebra of infinite rank over an arbitrary field of characteristic different from $2$. Dauletiyarova, Abdukhalikov, Sartayev~\cite{dauletiyarova2024free} found a linear basis for free metabelian Novikov algebra. 

Working over an infinite field with $p\neq 2$, Diniz, Gonçalves, da Silva and Souza~\cite{diniz2023two} described a finite generating set for the T-ideal of polynomial identities for every two-dimensional Jordan algebra and linear bases for the corresponding relatively free algebras were also given. In
\cite{diniz2024isomorphism} they showed that any two-dimensional Jordan algebras over a finite field $\FF$ with $p\neq2$ are isomorphic if and only if they satisfy the same polynomial identities. Moreover, a finite generating set for the T-ideal of polynomial identities for every two-dimensional Jordan algebra was determined in~\cite{diniz2024isomorphism} when $\FF$ is finite with $p\neq2$ and linear bases for the corresponding relatively free algebras were also given.

Recently, Dotsenko,  Ismailov and Umirbaev~\cite{dotsenko2023polynomial} proved that a Novikov algebra $\mathcal N$ over a field of characteristic zero satisfies a non-trivial polynomial identity (in the variety of Novikov algebras) if and only if the associate Lie algebra ${\mathcal N}^{(-)}$ is solvable. Moreover, they also established the Specht property for  the variety of Novikov algebras over a field of characteristic zero. Recall that a variety $\mathcal V$ of algebras is called \textit{Spechtian} or \textit{has the Specht property} if every T-ideal of polynomial identities for $\mathcal{V}$ is finitely generated as a T-ideal. For more details and history on Specht property for nonassociative algebras, we refer to~\cite{ismailov2024variety}. In~\cite{DotsenkoZhakhayev_2024} some properties of one-parameter family $\mathbf{N}^{\ell}_6$ were studied. In particular, it was proven that if two algebras from $\mathbf{N}^{\ell}_6$ satisfy the same polynomial identities, then they are isomorphic. 

Giambruno, Mishchenko, and Zaicev~\cite{giambruno2007codimension} established that
the growth of the codimension sequence $c_{n}(\A)$ for a two-dimensional nonassociative algebra $\A$ over a field of characteristic zero 
is bounded by $n+1$ or it grows exponentially as $2^n$ (see Section~\ref{section_codim} for the definitions). Namely, they proved that
$$c_n(\A)\leqslant n+1 \text{ or } \frac{1}{n^2}2^n\leqslant c_n(\A)\leqslant 2^n.$$
They also constructed a family of two-dimensional algebras $\{\A_{\alpha}\}_{\alpha\in\QQ}$ satisfying equality $c_n(\A_{\alpha})=n+1$ and having pairwise different T-ideals of polynomial identities.

Another problem, closely related to polynomial identities of an algebra $\A$, is the description of an image of $\A$ with respect to a multilinear polynomial $f$ (see Section~\ref{section_image} for the definitions). The generalized L'vov--Kaplansky conjecture claims that the multilinear image of any finite-dimensional simple algebra is always a vector space. Originally, L'vov--Kaplansky conjecture was formulated for matrices, but it was solved only for $2\times2$ matrices by Kanel-Belov, Malev, Rowen~\cite{Belov_Malev_Rowen_2012}. The generalized L'vov--Kaplansky conjectureis was established for the algebra
of real quaternions by Malev~\cite{Malev_2021}, the algebra of real octonions by Kanel-Belov, Malev, Pines, Rowen~\cite{Belov_Malev_Pines_Rowen_2024}, the Jordan algebra of a nondegenerated bilinear form by Malev, Yavich, Shayer~\cite{Malev_Yavich_Shayer_2022}. The image of null-filiform Leibniz algebras was described by de Mello, Souza~\cite{deMello_Souza_2023} for any multihomogeous polynomial; in particular, the  multilinear image is a vector space.

%=================================================================
\subsection{Results}
In our paper we work over the complex numbers $\CC$. In Section~\ref{section_notations} key definitions are given together with the formulation of Theorem~\ref{theor-bai-meng} classifying Novikov algebras of dimension two. In Section~\ref{section-associative case}, we describe generators for the T-ideals of polynomial identities for all two-dimensional associative Novikov algebras as well as linear bases for the corresponding relatively free algebras (see Propositions~\ref{prop-T2}, \ref{prop-N1-N2-N3}, \ref{prop-N4}). In Sections~\ref{section-nonassociative-T3}--\ref{section-nonassociative-N5}, we describe generators for the T-ideals of polynomial identities for all two-dimensional nonassociative Novikov algebras as well as linear bases for the corresponding relatively free algebras (see Theorems~\ref{theorem-T3}, \ref{theorem-N6}, \ref{theorem-N5}).

As corollaries from our results, in Corollary~\ref{cor_isomorphism} we establish that polynomial identities separate two-dimensional Novikov algebras, which are not associative. Namely, any two-dimensional Novikov algebras, which are not associative, are isomorphic if and only if they satisfy the same polynomial identities.  Note that Novikov algebras $\mathbf{N}_5$ and $\mathbf{N}_6^{\ell}$ have the same bases for relatively free algebras, but satisfy different polynomial identities. We also obtain the codimension sequences of all these algebras in Corollaries~\ref{cor-codimension-associative} and \ref{cor-codimension-nonassociative}. In particular, every two-dimensional Novikov algebra has at most linear growth of its  codimension sequence.  In Theorem~\ref{theo_image} we explicitly describe multilinear images of every two-dimensional Novikov algebra. In particular, we show that these images are vector spaces (see Corollary~\ref{cor_image}).

%=================================================================
%=================================================================
\section{Auxilliaries}\label{section_notations}

In the rest of the paper we work over the field $\FF=\CC$ of complex numbers. 
We write $\CC\langle x_1,\ldots, x_n \rangle$ for the free 
nonassociative non-unital $\CC$-algebra with free generators
$x_1, \ldots , x_n$. 
By $\CCX$ we denote the free nonassociative algebra on the infinite and enumerable set~$X=\{x_1, x_2, \ldots \}$
of free generators.

Recall that a polynomial 
$f(x_1, \ldots, x_n)$ in $\CCX$ is a \textit{polynomial identity}
for an
algebra $\A$ if $f(a_1, \ldots, a_n) = 0$ for all $a_1, \ldots, a_n\in \A$.
If $\A$ satisfies 
a nontrivial polynomial
identity, we call $\A$ an
\textit{algebra with polynomial identity}. The set 
$\Id_{\CC}(\A) = \Id(\A)$ of
all polynomial identities for $\A$ is a 
T-ideal, i.e., $\Id(\A)$ is an ideal of 
$\CCX$ such that
$\phi(\Id(\A))\subseteq \Id(\A)$ for every endomorphism $\phi$ of $\CCX$.
A T-ideal $I$ of $\CCX$ generated polynomials 
$f_1, \ldots, f_{k} \in \CCX$ 
is the minimal T-ideal of $\CCX$ that contains 
$f_1, \ldots, f_k$. 
Denote
$I = \Id(f_1, \ldots , f_k)$. We say that $f\in \CCX$ is a consequence of $f_1, \ldots, f_k$ if $f\in I$. A non-empty product of some letters from $\{x_1,x_2,\ldots,\}$ is called a monomial. Given
a monomial $w$ in $\CC\langle x_1,\ldots, x_n \rangle$, we write $\dg_{x_i}(w)$ for the number of letters $x_i$
in $w$ and
$\mdeg(w)\in \NN^n$ for the multidegree 
$(\dg_{x_1}(w), \ldots , \dg_{x_n}(w))$ of $w$, where $\NN=\{0,1,2, \ldots\}$.
An element $f\in \CCX$ is called (multi)homogeneous if it is a linear combination of monomials of the same (multi)degree. 
If $f\in \CCX$ is multihomogeneous of multidegree
$1^{n} = (1,\ldots , 1)$ ($n$ times), then $f$ is called \textit{multilinear}. Given $f\in\CCX$, we have $f=f_1+\cdots +f_k$ for some   unique multihomogeneous $f_1,\ldots, f_k$, which are called multihomogeneus components of $f$. Given  $\un{\de}=(\de_1,\ldots,\de_n)\in\NN^n$, we denote $|\un{\de}|=\de_1 + \cdots + \de_n$ and define a function $\un{\de}|\cdot|:\{1,\ldots,|\un{\de}|\}\to \{1,\ldots,n\}$ by
$$\un{\de}|r|=k \text{ if and only if }\de_1+\cdots+\de_{k-1}+ 1\leqslant r \leqslant \de_1+\cdots+\de_{k},$$
for all $1\leqslant r\leqslant |\un{\de}|$.

Assume $f\in \CCX$ is multihomogeneous of multidegree $\un{\de}\in\NN^n$. Given $1\leqslant i\leqslant n$ and $\un{\ga}\in\NN^k$ for some $k>0$ with $|\un{\ga}|=\de_i>0$, the {\it partial linearization} $\lin_{x_i}^{\un{\ga}}(f)$ of $f$ of multidegree $\un{\ga}$ with respect to $x_i$ is the multihomogeneous component of 
$$f(x_1,\ldots,x_{i-1},x_{i}+\cdots+x_{i+k-1},x_{i+k},\ldots,x_{n+k-1})$$ 
of multidegree $(\de_1,\ldots,\de_{i-1},\ga_{1},\ldots,\ga_{k},\de_{i+1},\ldots,\de_{n}$). As an example, 
\begin{multline*}
\rm{lin}_{x_2}^{(2,1)}\biggl((x_1 (x_1 x_2 )) (x_2 (x_2 x_3))\biggr) = \\
((x_1 (x_1 x_2 )) (x_2 (x_3 x_4)) + ((x_1 (x_1 x_2 )) (x_3 (x_2 x_4)) + ((x_1 (x_1 x_3 )) (x_2 (x_2 x_4)).
\end{multline*}

We will use the following trivial remark without reference to it. 

\begin{remark} Assume that an ideal $\I$ of $\CCX$ is generated by some multihomogeneous elements and $f\in\CCX$ is multihomogeneous of multidegree $\un{\de}$. Let $f-h\in\I$ for some $h\in \CCX$. Then there are unique $h_1,h_2\in\CCX$ such that $h=h_1+h_2$, where $h_1$ is multihomogeneous with $\mdeg(h_1)=\un{\de}$, the multihomogeneous component of $h_2$ of multidegree $\un{\de}$ is zero and $h_2\in\I$.  
\end{remark}

The result of subsequent applications of partial linearizations to $f$ is also called a partial linearization of $f$. The {\it complete linearization} $\lin(f)$ of $f$ is the result of subsequent applications of $\lin_{x_1}^{1^{\de_1}},\ldots, \lin_{x_n}^{1^{\de_n}}$ to $f$. Since $\CC$ is infinite, partial linearizations of an polynomial identity for an algebra $\A$ are polynomial identities for $\A$. The \textit{symmetric group} on the set $\{1, 2, \ldots, n\}$ is denoted by $\Sym_n$.

%Denote by $R_x$ and $L_x$ the operators of right and left multiplication, respectively, by an element $x$, i.e. $aR_x=ax$ and $aL_{x}=xa$. We use $T_x$ as a common notation for $R_x$ and $L_x$.   
%The \textit{symmetric group} on the set $\{1, 2, \ldots, n\}$ is denoted by $\mathcal{S}_n$.

The sum of the coefficients of a $f=\sum_{i} \al_i u_i\in\CCX$, where $\al_i\in\CC$ and $u_i$ is a monomial, we denote by $s(f)=\sum_i \al_i$.  

For $f_1, \ldots, f_n \in \CCX$, we write $f_1f_2\cdots f_n:= (\cdots((f_1f_2)f_3)\cdots)f_n$.  
If we need to be more specific, then we will use the following
notation for the left-bracket product of letters $(x_{i_1}\cdots x_{i_n})_{L} = (\cdots ((x_{i_1} x_{i_2})x_{i_3})\cdots )x_{i_n}$. In case $i_r=\cdots=i_s=j$ for some $1\leqslant r<s\leqslant n$, for short, we write 
$$(x_{i_1}\cdots x_{i_{r-1}} x_{j}^{s-r+1} x_{i_{s+1}} \cdots x_{i_n})_{L} = (x_{i_1}\cdots x_{i_n})_{L}.$$
Similarly, denote the right-bracket product of letters by $(x_{i_1}\cdots x_{i_n})_{R}$. For short, we introduce following notations for some monomials from $\CCX$, which play an important role in our paper:
\begin{eq}\label{eq_notations}
\begin{array}{ll}
v_{i,n}=v_{i,n}(x_1,\ldots,x_n) =  x_{i}x_{1} \cdots \cdot x_{i-1}  x_{i+1} \cdots x_n, & \text{ where } 1\leqslant i\leqslant n,\; n\geqslant 1, \\
w_n=w_n(x_1,\ldots,x_n)=(x_{1} (x_{2} x_{3})) x_{4} \cdots x_{n},
& \text{ where } n\geqslant 3. \\
\end{array}
\end{eq}

Throughout the paper
we will use the following result obtained by Bai and Meng~\cite{bai2001classification}.

%-----------------------------------------------------------------
\begin{theorem}[\!\!\cite{bai2001classification}]
\label{theor-bai-meng}
When the field is $\mathbb C$ any two-dimensional
Novikov algebra is isomorphic to one of the
following algebras:
\begin{equation*}
    \mathbf{T}_1,\mathbf{T}_2, \mathbf{T}_3, \mathbf{N}_1, \mathbf{N}_2, \mathbf{N}_3, \mathbf{N}_4, \mathbf{N}_5, \mathbf{N}^{\ell}_6 \;\;( \text{for }\ell \in\CC\backslash\{0,1
    \}),
\end{equation*}
where the algebras have basis 
$\{e_1, e_2\}$ and the multiplication tables are given below:
\[
\begin{array}{lllll}
\mathbf{T}_1: & e_1 e_1=0, & e_1 e_2=0, & e_2 e_1=0, & e_2 e_2=0 ;\\
\mathbf{T}_2: & e_1 e_1=e_2, & e_1 e_2=0, & e_2 e_1=0, & e_2 e_2=0 ; \\
\mathbf{T}_3: & e_1 e_1=0, & e_1 e_2=0, & e_2 e_1=-e_1, & e_2 e_2=0 ; \\
\mathbf{N}_1: & e_1 e_1=e_1, & e_1 e_2=0, & e_2 e_1=0, & e_2 e_2=e_2 ; \\
\mathbf{N}_2: & e_1 e_1=e_1, & e_1 e_2=0, & e_2 e_1=0, & e_2 e_2=0 ;
\\
\mathbf{N}_3: & e_1 e_1=e_1, & e_1 e_2=e_2, & e_2 e_1=e_2, & e_2 e_2=0 ;
\\
\mathbf{N}_4: & e_1 e_1=0, & e_1 e_2=e_1, & e_2 e_1=0, & e_2 e_2=e_2 ;
\\
\mathbf{N}_5: & e_1 e_1=0, & e_1 e_2=e_1, & e_2 e_1=0, & e_2 e_2=e_1+e_2;
\\
\mathbf{N}^{\ell}_6: & e_1 e_1=0, & e_1 e_2=e_1, & e_2 e_1=\ell e_1, & e_2 e_2=e_2.
\end{array}
\]
\end{theorem}

Given a T-ideal $\I$, with abuse of notations we denote the elements of $\CCX / \I$ by the same manner as elements of $\CCX$ (as an example, see Proposition~\ref{prop-T2} and its proof). Since $\CC$ is infinite, the following remark is trivial.

%-----------------------------------------------------------------
\begin{remark}\label{remark_key}
Assume $\A$ is an algebra, $\I$ is a T-ideal in $\Id(\A)$, $\B=\{f_i\,|\,i\in I\}$ is a subset of $\CCX$ of multihomogeneous elements  such that $\CCX / \I$ is a $\CC$-span of $\{f+\I \,|\, f\in\B\}$.  Then
\begin{enumerate}
\item[(a)] the image $\ov{\B}$ of $\B$ in $\CCX / \Id(\A)$ is its basis if and only if for every multidegree $\un{\de}\in\NN^n$, $n>0$, with non-empty set 
$I_{\un{\de}} = \{ i\in I\,|\, \mdeg(f_i)=\un{\de}\}$, if 
$$\sum_{i\in I_{\un{\de}}} \al_i f_i \in \Id(\A),\;\; \text{ where }\al_i\in \CC,$$
then $\al_i=0$ for all $i\in I_{\un{\de}}$; 

\item[(b)] if $\ov{\B}$ is a basis for $\CCX / \Id(\A)$, then $\Id(\A)=\I$.
\end{enumerate}
\end{remark}

Given and algebra $\A$, we write $\A^{\rm op}$ for the {\it opposite} algebra, i.e., as the vector space $\A^{\rm op}$ coincides with $\A$ and the multiplication $\ast$ is given on $\A^{\rm op}$ by $a\ast b= ba$ for all $a,b\in \A^{\rm op}$.

%=================================================================
%=================================================================
\section{Two-dimensional associative  Novikov algebras}\label{section-associative case}

%-----------------------------------------------------------------
\begin{prop}\label{prop-T2}
The algebra 
$\mathbf{T}_2$ 
is commutative, associative and nilpotent of index $3$. The ideal   
$\Id(\mathbf{T}_2)$
is minimally generated, as a T-ideal, by the polynomials
\begin{equation}\label{T-ideal(T2)}
    [x_1, x_2] \quad \text{and} \quad
(x_1 x_2)x_3.
\end{equation}
Moreover, the image of the following subset of $\CCX$ 
\begin{equation}\label{span-T2}
x_i, \;  \quad x_j x_k \quad 
(1\leqslant i,\; 1\leqslant j \leqslant k)
\end{equation}
in the relatively free algebra  $\CCX/\Id(\mathbf{T}_2)$ is its basis.
\end{prop}
\begin{proof}
By the multiplication table of~$\mathbf{T}_2$, it is easy see that
$\mathbf{T}_2$ satisfies the polynomial identities~\eqref{T-ideal(T2)}.
Denote by $\I$ the T-ideal generated by polynomial identities~\eqref{T-ideal(T2)}. The minimality of generating set~\eqref{T-ideal(T2)} for the T-ideal $\I$ is obvious. It is clear that $\CCX/\mathcal{I}$
is spanned by the image of set~\eqref{span-T2}.  Since $\I\subseteq \Id(\mathbf{T}_2)$, the vector space
$\CCX/\Id(\mathbf{T}_2)$ is spanned by the image of set~\eqref{span-T2}. The linear independence of the image of set~\eqref{span-T2} in $\CCX/\Id(\mathbf{T}_2)$ follows from part (a) of Remark~\ref{remark_key} and  $e_1^2\neq0$. The proof is concluded by part (b) of Remark~\ref{remark_key}.
\end{proof}

%-----------------------------------------------------------------
\begin{prop}\label{prop-N1-N2-N3} An algebra $\mathcal{A}\in \{\mathbf{N}_1, \mathbf{N}_2,\mathbf{N}_3\}$ is commutative and associative. The ideal $\Id(\mathcal{A})$ is minimally generated, as a T-ideal, by the polynomials
\begin{equation}\label{T-ideal-N123}
    [x_1, x_2] \quad \text{and} \quad
(x_1, x_2,x_3).
\end{equation}
Moreover, the image of the following subset of $\CCX$ 
\begin{equation}\label{span-N123}
    x_{i_1}x_{i_2} \cdots x_{i_n} \quad (n\geqslant 1,\;\; 1\leqslant i_1\leqslant \cdots \leqslant i_n)
\end{equation}
in the relatively free algebra $\CCX/\Id(\mathcal A)$ is its basis.
\end{prop}
\begin{proof} Assume 
$\mathcal{A}\in \{\mathbf{N}_1, \mathbf{N}_2,\mathbf{N}_3\}$.
By the multiplication table of $\mathcal{A}$, the algebra $\mathcal{A}$ is
commutative and associative. Let $\I$ be the T-ideal generated by polynomial identities~\eqref{T-ideal-N123}. The minimality of generating set~\eqref{T-ideal-N123} for the T-ideal $\I$ is obvious. Since the commutator and associator belong
to $\I$, we have that $\CCX/\I$ is spanned by the image of set \eqref{span-N123}.
Note that
$\I\subseteq \Id(\A)$, thus, 
$\CCX/\Id(\A)$ is also spanned by the image of set~\eqref{span-N123}.
The linear independence of the image of set~\eqref{span-N123} in $\CCX/\Id(\A)$
follows from part (a) of Remark~\ref{remark_key} and $e_1^n\neq0$. Therefore, the required statement follows from part (b) of Remark~\ref{remark_key}.
\end{proof}

%-----------------------------------------------------------------
\begin{prop}\label{prop-N4}
The ideal $\Id(\mathbf{N}_4)$
is minimally generated, as a T-ideal, by the polynomials
\begin{equation}\label{T-ideal-N4}
(x_1, x_2, x_3) \quad
\text{and}
\quad
    x_1[x_2, x_3].
\end{equation}
Moreover, the image of the following subset of $\CCX$ 
\begin{equation}
\label{span-N4}
  x_{i}x_{j_1} \cdots x_{j_{n-1}} \quad (n\geqslant 1,\;\; i\geqslant 1,\;\;  1\leqslant j_{1}\leqslant j_2 \leqslant \cdots \leqslant j_{n-1})
\end{equation}
in the relatively free algebra $\CCX/\Id(\mathbf{N}_4)$ is its basis.
\end{prop}
\begin{proof}
By the multiplication table of~$\mathbf{N}_4$, it is easy to see that
$\mathbf{N}_4$ satisfies the polynomial identities~\eqref{T-ideal-N4}.
Denote by $\I$ the T-ideal generated by polynomial identities~\eqref{T-ideal-N4}. The minimality of generating set~\eqref{T-ideal-N4} for the T-ideal $\I$ follows from the facts that the $\CC$-span of $(x_{\si(1)},x_{\si(2)},x_{\si(3)})$, $\si\in\Sym_3$, does not contain $x_1[x_2,x_3]$ and the $\CC$-span of $x_{\si(1)}[x_{\si(2)},x_{\si(3)}]$, $\si\in\Sym_3$, does not contain  $(x_1,x_2,x_3)$.

It is clear that $\CCX/\mathcal{I}$
is spanned by the image of set~\eqref{span-N4}.  Since $\I\subseteq \Id(\mathbf{N}_4)$, the vector space
$\CCX/\Id(\mathbf{N}_4)$ is spanned by the image of set~\eqref{span-N4}. 

To prove the linear independence of the image of set~\eqref{span-N4} in $\CCX/\Id(\mathbf{N}_4)$, by part (a) of Remark~\ref{remark_key} it is enough to show that for every multidegree $\un{\de}=(\de_1,\ldots,\de_n)$ with $\de_1,\ldots,\de_n>0$ and $m=|\un{\de}|>1$ if 
$$f(x_1,\ldots,x_n) = \sum_{i=1}^n \al_k\, \big(x_{i}x_{1}^{\de_1} \cdots \cdot x_{i-1}^{\de_{i-1}} x_i^{\de_i-1} x_{i+1}^{\de_{i+1}} \cdots x_n^{\de_n}\big)_{L}\, \in \Id(\mathbf{N}_4),$$
where $\al_1,\ldots,\al_n\in\CC$, then 
\begin{eq}\label{eq_al}
\al_1=\cdots=\al_n=0.
\end{eq}%
Applying polynomial identity $x_1[x_2,x_3]$ to the complete linearization of $f$, we obtain the following multilinear polynomial identity for $\mathbf{N}_4$ of multidegree $1^m$:
$$h(x_1,\ldots,x_m)=\de_1 ! \cdots \de_{n}!\sum_{r=1}^m \frac{1}{\de_{\un{\de}|r|}} \al_{\un{\de}|r|}\, 
x_r x_1\cdots x_{r-1} x_{r+1} \cdots x_m.$$

\noindent{}Claim~\eqref{eq_al} follows from equalities 
%$$h(e_2,\ldots,e_2)=\de_1 ! \cdots \de_{n}!(\al_1+\cdots+\al_n)=0,$$
$$h(e_2,\ldots,e_2,\underbrace{e_1}_{r^{\rm th}\,\text{position}},e_2,\ldots,e_2)=
\de_1 ! \cdots \de_{n}! \frac{1}{\de_{\un{\de}|r|}} \, \al_{\un{\de}|r|}\, e_1=0$$% 
for all $1\leqslant r\leqslant m$, since $e_1e_2^{m-1}=e_1$ and $e_2^{i} e_1 e_2^{m-i-1}=0$ in case $1\leqslant i<m$. The proof is concluded by part (b) of Remark~\ref{remark_key}.
\end{proof}

The following remark shows that the straightforward approach to obtain a basis of a relatively free algebra from a basis of its multilinear components does not work in general. That is why we have to consider a non-multilinear case in the proof of Proposition~\ref{prop-N4} as well as in the proofs of Theorems~\ref{theorem-T3}, \ref{theorem-N6}, \ref{theorem-N5} (see below).

%-----------------------------------------------------------------
\begin{example}\label{example_1}
Consider the T-ideal $\I$ generated by $(x_1,x_2,x_3)$ and $x_1^3$. Then a basis of a multilinear component of $L=\CCX / \I$ of multidegree $1^4$ is
\begin{eq}\label{eq_basis}
\begin{array}{l}
x_1x_2x_3x_4,\; x_1x_2x_4x_3,\; x_1x_3x_2x_4,\; x_1x_3x_4x_2\;, x_1x_4x_2x_3, \\
x_2x_1x_3x_4,\; x_2x_1x_4x_3,\; x_2x_3x_1x_4,\; x_2x_3x_4x_1,\; x_2x_4x_1x_3, \\
x_3x_1x_2x_4,\; x_3x_4x_1x_2. \\
\end{array}
\end{eq}%
\noindent{}(see Proposition 2 of~\cite{Lopatin_2005}). How can we try to obtain a basis of multihomogeneous component $L_{(31)}$ of $L$ of multidegree $(3,1)$ from basis~\eqref{eq_basis}? The straightforward approach is to make substitutions $x_i\to x_1$, $x_4\to x_2$ ($1\leqslant i\leqslant 3$) in basis~\eqref{eq_basis} and then eliminate elements, which appear more than one time. Then the result will be
$$x_1^3 x_2,\; x_1^2 x_2 x_1,\; x_1 x_2 x_1^2,$$
but a basis for $L_{(31)}$ is $\{x_1^2 x_2 x_1\}$ (see Proposition 1 of~\cite{Lopatin_2005}). 
\end{example}

%=================================================================
%=================================================================
\section{Nonassociative case: Algebra $\mathbf{T}_3$}\label{section-nonassociative-T3}

Recall that the multiplication on the algebra $\mathbf{T}_3$ is defined as follows:
\[
\mathbf{T}_3:  e_1 e_1=0,\qquad
e_1 e_2=0,  \qquad
e_2 e_1=-e_1, \qquad
e_2 e_2=0.
\]
Note that the algebra $\mathbf{T}_3$ is not associative, since $(e_2,e_2,e_1)=-e_1\neq0$.

%-----------------------------------------------------------------
\begin{theorem}\label{theorem-T3}
The ideal $\Id(\mathbf{T}_3)$ is minimally generated, as a T-ideal, by the polynomials
\begin{equation}
\label{T-ideal-T3}
    (x_1 x_2) x_3 \quad
\text{and}
\quad
    x_1(x_2x_3)-x_2(x_1x_3).
\end{equation}
Moreover, the image of the following subset of $\CCX$ 
\begin{equation}
\label{span-T3}
    (x_{j_{n-1}}\cdots x_{j_1}x_{i})_{R}  \quad (\text{where}\;\; n\geqslant 1,\;\; i\geqslant 1,\;\; 
    1\leqslant j_{1}\leqslant j_2 \leqslant \cdots \leqslant j_{n-1})
\end{equation}
in the relatively free algebra $\CCX/\Id(\mathbf{T}_3)$ is its basis.
\end{theorem}
\begin{proof}By the multiplication table of~$\mathbf{T}_3$, it is easy to see that
$\mathbf{T}_3$ satisfies the polynomial identities~\eqref{T-ideal-T3}. Denote by $\I$ be the T-ideal of $\CCX$ generated by set~\eqref{T-ideal-T3}.  The minimality of generating set~\eqref{T-ideal-T3} for T-ideal $\I$ follows from the facts that the $\CC$-span of $(x_{\si(1)}x_{\si(2)})x_{\si(3)}$, $\si\in\Sym_3$, does not contain $x_1(x_2x_3)-x_2(x_1x_3)$ and the $\CC$-span of $x_{\si(1)}(x_{\si(2)}x_{\si(3)})-x_{\si(2)}(x_{\si(1)}x_{\si(3)})$, $\si\in\Sym_3$, does not contain  $(x_1x_2)x_3$.

Since for every monomials $u,v\in\CCX$ with $\deg(u)>1$ the product $uv$ is zero in $\CCX/\mathcal{I}$, then it is easy to see that $\CCX/\mathcal{I}$ is spanned by  the image of set~\eqref{span-T3}. Since $\I\subseteq \Id(\mathbf{T}_3)$, the vector space
$\CCX/\Id(\mathbf{T}_3)$ is spanned by the image of set~\eqref{span-T3}. 

To prove the linear independence of the image of set \eqref{span-T3} in $\CCX/\Id(\mathbf{T}_3)$, by part (a) of Remark~\ref{remark_key} it is enough to show that for every multidegree $\un{\de}=(\de_1,\ldots,\de_n)$ with $\de_1,\ldots,\de_n>0$ and $m=|\un{\de}|>1$ if 
%$$
%f(x_1,\dots,x_n)=
%\sum_{k=1}^{n}\alpha_{k}\,
%x_{n}^{\de_n}\left( \cdots x_{k+1}^{\de_{k+1}}  (x_k^{\de_k}(x_{k-%1}^{\de_{k-1}} \cdots  (x_{2}^{\de_2}
%(x_1^{\de_1}x_{k}))\cdots))\cdots\right)  \in \Id(\mathbf{T}_3),
%
\begin{eq}
f(x_1,\dots,x_n)=
\sum_{i=1}^{n}\alpha_{i}\,
\big(x_{n}^{\de_n} \cdots x_{i+1}^{\de_{i+1}}  x_i^{{\de_i-1}}x_{i-1}^{\de_{i-1}} \cdots  x_{2}^{\de_2}x_1^{\de_1}x_{i}\big)_{R}  \in \Id(\mathbf{T}_3),
\end{eq}
where $\al_1,\ldots,\al_n\in\CC$, then 
\begin{eq}\label{eq_al_T3}
\al_1=\cdots=\al_n=0.
\end{eq}%

\noindent{}Applying polynomial identity $ x_1(x_2x_3)-x_2(x_1x_3)$ to the complete linearization of $f$, we obtain the following multilinear polynomial identity for $\mathbf{T}_3$ of multidegree $1^m$:
$$h(x_1,\ldots,x_m)=\de_1 ! \cdots \de_{n}!\sum_{r=1}^m \frac{1}{\de_{\un{\de}|r|}} \al_{\un{\de}|r|}\, 
(x_{m} \cdots x_{r+1} x_{r-1} \cdots  x_{2}x_1 x_{r})_{R}.$$
%$$h(x_1,\ldots,x_m)=\de_1 ! \cdots \de_{n}!\sum_{r=1}^m \frac{1}%{\de_{\un{\de}|r|}} \al_{\un{\de}|r|}\, 
%x_{m}\left( \cdots x_{r+1} (x_{r-1} \cdots  (x_{2}(x_1 x_{r}))\cdots)\cdots\right).$$

\noindent{}Claim~\eqref{eq_al_T3} follows from equalities
$$h(e_2,\ldots,e_2,\underbrace{e_1}_{r^{\rm th}\,\text{position}},e_2,\ldots,e_2)= \pm\,
\de_1 ! \cdots \de_{n}! \frac{1}{\de_{\un{\de}|r|}} \, \al_{\un{\de}|r|}\, e_1=0$$% 
for all $1\leqslant r\leqslant m$, since $(e_2^{m-1}e_1)_{R}=\pm e_1$ and $e_1e_2=0$. The proof is concluded by part (b) of Remark~\ref{remark_key}.
\end{proof}

%=================================================================
%=================================================================
\section{Nonassociative case: algebra $\mathbf{N}^{\ell}_6$}\label{section-nonassociative-N6}

Recall that the algebra $\mathbf{N}^{\ell}_6$
has the following multiplication table:
\[
\mathbf{N}^{\ell}_6:  e_1 e_1=0, \qquad e_1 e_2=e_1, \qquad e_2 e_1=\ell e_1,  \qquad e_2 e_2=e_2, \qquad \ell \neq 0,1.
\]

\noindent{}In this section we will prove the following result.

%-----------------------------------------------------------------
\begin{theorem}\label{theorem-N6}
\noindent{1.}    The ideal $\Id(\mathbf{N}^{\ell}_6)$ is minimally generated, as a T-ideal, by the polynomials
\begin{eq} \label{eq_15}
(x_1 x_2) x_3-(x_1 x_3) x_2, 
\end{eq}
\vspace{-0.8cm}
\begin{eq}\label{eq_16}
(x_1, x_2, x_3)-(x_2, x_1, x_3), 
\end{eq}
\vspace{-0.8cm}
\begin{eq} \label{eq_17}
(x_1,x_2,x_3)-(x_1,x_3,x_2)-\ell[x_3,x_2]x_1.
\end{eq}
\vspace{-0.8cm}

\medskip
\noindent{2.} 
Let $\B$ be the following subset of
$\CCX:$
\begin{enumerate}
    \item[(a)] $x_{i} x_{j_1}x_{j_2} \cdots x_{j_{n-1}} \qquad\qquad\;\, (n\geqslant 1,\;\;  1\leqslant j_{1}\leqslant j_2 \leqslant \cdots \leqslant j_{n-1},\;\; i\geqslant1)$,
    \item[(b)] $(x_{j_1} (x_{j_2} x_{j_3})) x_{j_4} \cdots  x_{j_n} \quad\quad (n\geqslant 3,\;\; 1\leqslant j_{1}\leqslant j_2 \leqslant \cdots \leqslant j_{n})$. 
    \end{enumerate}
\noindent{} Then the image $\ov{\B}$ of $\B$ in the relatively free algebra
$\CCX/\Id(\mathbf{N}^{\ell}_6)$ is its basis.
\end{theorem}

Introduce some notations. Elements of $\B$ from item (a) or (b), respectively, are said to be of {\it type} (a) or (b), respectively. We write $V$ for the $\CC$-span of $x_{i_1} \cdots x_{i_n}$ in $\CCX$ for all $n\geqslant1$ and $i_1,\ldots,i_n\geqslant1$. Denote by $W$ the $\CC$-span of $\B$ in $\CCX$. 
Let $\I$ be the T-ideal of $\CCX$ generated by polynomials~\eqref{eq_15}--\eqref{eq_17}. 
Introduce the following equivalence relation  $\equiv$ on $\CCX$, which is called the {\it congruence} modulo $\I$: for $f,g\in \CCX$ denote
$$
f\equiv g\quad\mathrm{if~and~only~if}\quad f+\I=g+\I.
$$%
Since polynomial~\eqref{eq_15} lies in $\I$, we obtain
\begin{equation}\label{eq_18}
(x_1 x_3) x_2 \equiv (x_1 x_2) x_3.
\end{equation}%

%----------------------------------------------------------------
\begin{lemma}\label{lemma_N6_min}
The generating set~\eqref{eq_15}--\eqref{eq_17} for the T-ideal $\I$ is minimal.
\end{lemma}
\begin{proof} Denote by $L_1\subset \CCX$ the $\CC$-span of  $(x_{\si(1)}x_{\si(2)})x_{\si(3)} - (x_{\si(1)}x_{\si(3)})x_{\si(2)}$, $\si\in\Sym_3$. Similarly, let $L_2$ be the $\CC$-span of $(x_{\si(1)}, x_{\si(2)}, x_{\si(3)})-(x_{\si(2)}, x_{\si(1)}, x_{\si(3)})$, $\si\in\Sym_3$, and let $L_3$ be the $\CC$-span of $(x_{\si(1)},x_{\si(2)},x_{\si(3)})-(x_{\si(1)},x_{\si(3)},x_{\si(2)})-\ell[x_{\si(3)},x_{\si(2)}]x_{\si(1)}$, $\si\in\Sym_3$.  Then the claim of the lemma follows from the following three claims.

\medskip
\noindent{\bf 1.} {\it The polynomial~\eqref{eq_15} does not belong to $L_2+L_3$.}

Assume the contrary. Then in the free associative algebra $\CCX / \Id( (x_1,x_2,x_3) )$ we have 
$$x_1x_2x_3 - x_1 x_3 x_2 = \al_1 [x_2,x_3] x_1 + \al_2 [x_1,x_3] x_2 + \al_3 [x_1,x_2]x_3$$
for some $\al_1,\al_2,\al_3\in\CC$. Considering the coefficients of $x_1x_2x_3$ and $x_2x_1x_3$, respectively, we obtain that $1=\al_3$ and $0=\al_3$, respectively; a contradiction.

\medskip
\noindent{\bf 2.} {\it The polynomial~\eqref{eq_16} does not belong to $L_1+L_3$.}

Assume the contrary. Then in the relatively free algebra $\CCX / \Id( (x_1,x_2)x_3 )$ we have 
$$-x_1 (x_2 x_3) + x_2 (x_1 x_3) = \al_1 x_1 [x_2,x_3] + \al_2 x_2 [x_1,x_3] + \al_3 x_3 [x_1,x_2]$$
for some $\al_1,\al_2,\al_3\in\CC$. Considering the coefficients of $x_1(x_2x_3)$ and $x_1(x_3x_2)$, respectively, we obtain that $-1=\al_1$ and $0=\al_1$, respectively; a contradiction.

\medskip
\noindent{\bf 3.} {\it The polynomial~\eqref{eq_17} does not belong to $L_1+L_2$.}

Assume the contrary. Since $\ell\neq0$, in the free associative algebra $\CCX / \Id( (x_1,x_2,x_3) )$ we have 
$$ [x_3, x_2] x_1 = \al_1 x_1 [x_2,x_3] + \al_2 x_2 [x_1,x_3] + \al_3 x_3 [x_1,x_2]$$
for some $\al_1,\al_2,\al_3\in\CC$. Considering the coefficients of $x_3 x_2 x_1$ and $x_3 x_1 x_2$, respectively, we obtain that $1=-\al_3$ and $0=\al_3$, respectively; a contradiction.
\end{proof}

%----------------------------------------------------------------
\begin{remark}\label{remark_newN6}
Assume that $i,j,k\in\{1,2\}$. Then $[e_i,e_j]\neq 0$ if and only if $\{i,j\}=\{1,2\}$; in this case, we have $[e_1,e_2]=-[e_2,e_1]=(1-l)e_1$. Similarly, $(e_i,e_j,e_k)\neq0$ if and only if $(i,j,k)=(2,2,1)$; in this case, we have $(e_2,e_2,e_1)=\ell(1-\ell)e_1$. In particular, $\mathbf{N}^{\ell}_6$ is nonassociative.
\end{remark}

%In this section we will prove the following result.

%-----------------------------------------------------------------
%\begin{theorem}
%\noindent $1.$
%    The ideal $\Id(\mathbf{N}^{\ell}_6)$ is generated, as a T-ideal, by the polynomials
%\begin{equation}\label{T-ideal-N6}
%\begin{gathered}
%\left(x_1 x_2\right) x_3-\left(x_1 x_3\right) x_2, \quad\left(x_1, x_2, x_3\right)-\left(x_2, x_1, x_3\right), \qquad
%x_1[x_2,x_3]=\ell[x_2,x_3]x_1, \\
%U\left(x_1, x_2, x_3, x_4\right):=\left(x_1\left(x_2 x_3\right)\right) x_4-\left(x_1\left(x_4 x_3\right)\right) x_2 .
%\end{gathered}
%\end{equation}
%\noindent $2.$
%    Let $\B$ be the following set from
%$\CCX:$
%\begin{enumerate}
%    \item[(a)] $x_i,$
%    \item[(b)]  $\left(\cdots((x_{i}x_{j_1}) x_{j_2})x_{j_3} \cdots \right)x_{j_{n}}  \quad (n\geqslant 1);$
%    \item[(c)] $(x_{j_1}, x_{j_2}, x_{i})R_{x_{j_3}} \ldots R_{x_{j_{n-1}}} \quad (n\geqslant3)$,
%\end{enumerate}
%where $1\leqslant i$ and
%$1< j_1< j_{2}<\ldots < j_{n}$.
%The set $\B$ is a basis for the relatively free algebra
%$\CCX/\Id(\mathbf{N}^{\ell}_6)$.
%\end{theorem}

%We consider the polynomial
%\begin{equation}\label{identity-U-N6}
%U(x_1,x_2,x_3,x_4)=(x_{1}(x_2x_3))x_4-(x_1(x_4x_3))x_2.
%\end{equation}
 
%-----------------------------------------------------------------
\begin{lemma}\label{lemma_identities-N6}
The algebra $\mathbf{N}^{\ell}_6$ satisfies polynomial identity~\eqref{eq_17}, i.e.,  $\I\subseteq  \Id(\mathbf{N}^{\ell}_6)$.
\end{lemma}
\begin{proof}
It is easy to see that $\CC$-$ \mathrm{span}\{e_1\}$ 
is an ideal of $\mathbf{N}_{6}^{\ell}$ with null product  and
$\CC$-$ \mathrm{span}\{e_2\}$ is a subalgebra of $\mathbf{N}_{6}^{\ell}$
generated by an idempotent element. In particular, the algebra $\CC$-$ \mathrm{span}\{e_2\}$  is associative, commutative with non-zero multiplication. It is easy to see that these observations together with Remark~\ref{remark_newN6} conclude the proof.
\end{proof}

%-----------------------------------------------------------------
\begin{lemma}\label{lemma_identities_I}
The following polynomials lay in the T-ideal $\I$:
\begin{align}
([x_1,x_2],x_3,x_4), \label{eq_19} 
\\
(x_{1}x_{2},x_{3},x_4) -  (x_1,x_3,x_4)x_2,\label{eq_20}
\\
 (x_1,x_2,x_3)x_4 - (x_4,x_2,x_3)x_1, \label{eq_21}
    \\
    x_1(x_2,x_3,x_4) - \ell(x_2,x_3,x_4)x_1. \label{eq_22}
\end{align}
\end{lemma}
\begin{proof} By congruence~\eqref{eq_18}, $(x_1,x_2,x_3)-(x_1,x_3,x_2)\equiv  - x_1[x_2,x_3]$. Hence, polynomial~\eqref{eq_17} implies that
\begin{eq}\label{eq_star}
x_1[x_2,x_3] \equiv \ell [x_2,x_3]x_1.
\end{eq}%

\noindent{}Consequently applying congruences~\eqref{eq_star} and~\eqref{eq_18}, we have $([x_1,x_2],x_3,x_4)=([x_1,x_2]x_3)x_4 - [x_1,x_2](x_3x_4) \equiv \ell^{-1}(x_3[x_1,x_2])x_4 - \ell^{-1} (x_3x_4)[x_1,x_2] \equiv 0$, i.e., polynomial~\eqref{eq_19} lies in the T-ideal $\I$.

It is well-known that polynomial identity~\eqref{eq_20} holds in every right commutative algebra. Namely, $0=\Big((x_1x_2)x_3)x_4 - ((x_1x_2)x_3)x_4 \Big) + \Big(  (x_1 x_2)(x_3x_4) - (x_1 x_2)(x_3x_4) \Big) \equiv
\Big( ((x_1x_2)x_3)x_4 -  (x_1 x_2)(x_3x_4)  \Big)  - \Big(  ((x_1x_3)x_4)x_2  -  (x_1 (x_3x_4))x_2 \Big)$ by congruence~\eqref{eq_18}. 

Consequently, applying polynomials~\eqref{eq_20},~\eqref{eq_19},~\eqref{eq_20}, we obtain 
 $(x_1,x_2,x_3)x_4 \equiv (x_1x_4,x_2,x_3)\equiv(x_4x_1,x_2,x_3)\equiv (x_4,x_2,x_3)x_1$, i.e., polynomial \eqref{eq_21} lies in $\I$.

The Teichm\"uller's identity (for example, see~\cite[pg. 343]{zhevlakov1982rings})
\begin{eq}\label{eq_Teich}
x_1(x_2,x_3,x_4)
=(x_1x_2,x_3,x_4)-(x_1,x_2x_3,x_4)+(x_1,x_2,x_3x_4) - (x_1,x_2,x_3)x_4
\end{eq}%
holds in any algebra, as a direct consequence of the definition of the associator. Applying polynomials from $\I$, we obtain
$$(x_1x_2,x_3,x_4)\stackrel{\eqref{eq_19}}{\equiv} (x_2x_1,x_3,x_4)\stackrel{\eqref{eq_20}}{\equiv} (x_2,x_3,x_4)x_1,$$
$$(x_1,x_2x_3,x_4) \stackrel{\eqref{eq_16}}{\equiv}  (x_2 x_3,x_1,x_4) \stackrel{\eqref{eq_20}}{\equiv}  
(x_2,x_1,x_4)x_3\stackrel{\eqref{eq_16}}{\equiv} (x_1,x_2,x_4)x_3 \stackrel{\eqref{eq_21},\eqref{eq_16}}{\equiv}
(x_2,x_3,x_4)x_1,$$
$$(x_1,x_2,x_3x_4) \stackrel{\eqref{eq_17}}{\equiv} (x_1,x_3x_4,x_2) + \ell [x_3x_4,x_2]x_1,$$
$$(x_1,x_2,x_3)x_4 \stackrel{\eqref{eq_21}}{\equiv} (x_4,x_2,x_3)x_1.$$%
Since
$$ (x_1,x_3x_4,x_2) \stackrel{\eqref{eq_16}}{\equiv} (x_3x_4, x_1,x_2) \stackrel{\eqref{eq_20}}{\equiv} (x_3,x_1,x_2)x_4 \stackrel{\eqref{eq_16}}{\equiv} (x_1,x_3,x_2)x_4 \stackrel{\eqref{eq_21}}{\equiv} $$ 
$$ (x_4,x_3,x_2)x_1 \stackrel{\eqref{eq_17}}{\equiv} (x_4,x_2,x_3)x_1 + \ell ([x_2,x_3]x_4) x_1,$$
identity~\eqref{eq_Teich} implies that polynomial~\eqref{eq_22} lies in the T-ideal $\I$.

\end{proof}

Applying congruence~\eqref{eq_18} we obtain that $V +\I \subset W +\I$ in $\CCX/ \I$. Moreover,  polynomials~\eqref{eq_16}, \eqref{eq_17}, \eqref{eq_21}, respectively, from $\I$ imply that 
\begin{equation}\label{eq_lemma_N6_1}
x_2(x_1x_3) \equiv x_1(x_2 x_3) + f_1,
\end{equation}
\begin{equation}\label{eq_lemma_N6_2}
x_1(x_3x_2) \equiv  x_1(x_2 x_3) + f_2,
\end{equation}
\begin{equation}\label{eq_lemma_N6_3}
(x_4(x_2x_3))x_1 \equiv  (x_1(x_2 x_3)) x_4 + f_3,
\end{equation}
respectively, for some multihomogeneous $f_1,f_2,f_3\in V$ with $\mdeg(f_1)=\mdeg(f_2)=1^3$ and $\mdeg(f_3)=1^4$. Adding polynomials $(x_1,x_2x_3,x_4) - (x_2x_3,x_1,x_4)$ and $(x_2x_3,x_1,x_4) - (x_2,x_1,x_4)x_3$ from $\I$ (see polynomials~\eqref{eq_16} and~\eqref{eq_20}), we obtain that
\begin{equation}\label{eq_lemma_N6_4}
x_1((x_2 x_3) x_4)\equiv  (x_1(x_2x_3))x_4 - ((x_2x_1)x_4)x_3 + (x_2(x_1x_4))x_3.
\end{equation}
Polynomial~\eqref{eq_22} from $\I$ implies that 
$$x_1(x_2(x_3 x_4)) \equiv x_1((x_2 x_3)x_4) - \ell ((x_2 x_3)x_4) x_1 + \ell (x_2(x_3x_4))x_1.
$$
Applying congruence~\eqref{eq_lemma_N6_4} together with congruence~\eqref{eq_18} to the above congruence, we obtain 
\begin{equation}\label{eq_lemma_N6_5}
x_1(x_2(x_3 x_4)) \equiv  - (\ell+1) ((x_2 x_1)x_3) x_4 + (x_1(x_2x_3))x_4  + (x_2(x_1x_4))x_3  + \ell (x_2(x_3x_4))x_1.
\end{equation}

%-----------------------------------------------------------------
\begin{lemma}\label{lemma_N_claims}
Assume that $n\geqslant 2$ and for each monomial $w\in \CCX$ with $\deg(w)<n$ we have that $w\equiv w_0$ for some homogeneous $w_0\in W$ with $\deg(w)=\deg(w_0)$.  Let $u,v\in\CCX$ be monomials with $\deg(uv)= n$. Then 
$$uv\equiv h \text{ for some homogeneous }h\in W\text{ with }  \deg(h)=n $$% 
\noindent{}if one of the following conditions holds:
\begin{enumerate}
\item[(a)] $v=x_i$ for some $i\geqslant1$;

\item[(b)]  $\deg(u)\geqslant4$;

\item[(c)] $\deg(u)=3$;

\item[(d)] $\deg(u)=2$;

\item[(e)] $u=x_i$ for some $i\geqslant1$.
\end{enumerate}
\end{lemma}
\begin{proof}

\medskip
\noindent{\bf (a)} Without loss of generality, we can assume that $u\in \B$. If $u$ has type (a), then congruence~\eqref{eq_18} concludes the proof.  If $u$ has type (b), then congruence~\eqref{eq_18} together with congruences~\eqref{eq_lemma_N6_1}, \eqref{eq_lemma_N6_2}, \eqref{eq_lemma_N6_3} conclude the proof of part (a).

\medskip
\noindent{\bf (b)} Without loss of generality, we can assume that $u\in \B$. Since $\deg(u)\geqslant4$, then $u=u_0 x_i$ for some monomial $u_0$ and $i\geqslant1$. Thus,  $uv=(u_0 x_i)v \equiv (u_0 v)x_i$ by congruence~\eqref{eq_18}.  Part (a) concludes the proof.

\medskip
\noindent{\bf (c)} Without loss of generality, we can assume that $u\in \B$. Since $\deg(u)=3$, then $u=(x_ix_j) x_k$ or $u=x_i (x_j x_k)$ for some $i,j,k\geqslant1$. In the first case,  $uv\equiv ((x_ix_j) v) x_k$ by congruence~\eqref{eq_18} and  part (a) concludes the proof. In the second case,  $uv\equiv (v (x_j x_k))x_i + f'$ by congruence~\eqref{eq_lemma_N6_3}, where $f'$ is the result of substitution $x_l\to v$ in some homogeneous $f\in V$ for some $l\geqslant1$, where $\deg(f)=4$. Applying congruence~\eqref{eq_18} together with part (a) to $f'$ and applying part (a) to $(v (x_j x_k))x_i$ we conclude the proof. 

\medskip
\noindent{\bf (d)} We have $u=x_i x_j$ for some $i,j\geqslant1$. Then $uv\equiv (x_i v) x_j$ by congruence~\eqref{eq_18} and part (a) concludes the proof. 

\medskip
\noindent{\bf (e)} Without loss of generality, we can assume that $v\in \B$.  Then, one of the following five possibilities holds.
\begin{enumerate}
\item[($\rm e_1$)] $v=(v' x_j) x_k$ for some monomial $v'\in\CCX$ and $j,k\geqslant1$. Hence, congruence~\eqref{eq_lemma_N6_4} implies that $x_i v \equiv (x_i(v'x_j))x_k - ((v'x_i)x_k)x_j + (v'(x_ix_k))x_j$. Finally, part (a) conclude the proof.

\item[($\rm e_2$)] $n=2$ and $v=x_j$ for some  $j\geqslant 1$. Then the required statement holds.

\item[($\rm e_3$)] $n=3$ and $v=x_j x_k$ for some  $j,k\geqslant 1$. Then $uv=x_i (x_j  x_k)$ and  congruences~\eqref{eq_lemma_N6_1}, \eqref{eq_lemma_N6_2} together with   congruence~\eqref{eq_18} conclude the proof.

\item[($\rm e_4$)] $n=4$ and $v=x_{j_1} (x_{j_2} x_{j_3})$ for some  $1\leqslant j_{1}\leqslant \cdots \leqslant j_3$. Then 
$x_iv = x_i(x_{j_1} (x_{j_2} x_{j_3}))$ and  congruence~\eqref{eq_lemma_N6_5} together with part (a) imply the required statement. 

\item[($\rm e_5$)] $n=5$ and $v=(x_{j_1} (x_{j_2} x_{j_3})) x_{j_4}$ for some  $1\leqslant j_{1}\leqslant\cdots \leqslant j_4$. Then 
$x_iv = x_i(x_{j_1} (x_{j_2} x_{j_3}))x_{j_4})$ and we apply congruence~\eqref{eq_lemma_N6_4}  to $x_iv$, where we consider $x_{j_2} x_{j_3}$ instead of $x_3$. Parts (a) and (c) conclude the proof.
\end{enumerate}
\end{proof}

%-----------------------------------------------------------------
\begin{lemma}\label{lemma-N6}
The algebra $\CCX/ \I $ is the $\CC$-span of the set $\{f + \I \,|\, f\in \B \}$.
\end{lemma}
\begin{proof} 
The lemma is a consequence of the following claim: if  $w\in\CCX$ is a monomial of degree $n$, then $w\equiv w_0$ for some homogeneous $w_0\in W$ with $\deg(w)=\deg(w_0)$. We prove this claim by induction on $n$. 

If $n=1,2$, then $\deg(w)\leqslant2$ and $w\in W$.

Let $n\geqslant 3$ and for each monomial $w'\in \CCX$ with $\deg(w')<n$ we have that $w'\equiv w'_0$ for some homogeneous $w'_0\in W$ with $\deg(w')=\deg(w'_0)$. Note that $w=uv$ for some monomials $u,v\in\CCX$. By parts (b)--(e) of Lemma~\ref{lemma_N_claims} we have that the claim holds.
\end{proof}

%-------------------------------------------------------------------
\medskip
\begin{proof_of}{of Theorem~\ref{theorem-N6}.} Since $\I\subset \Id(\mathbf{N}_6^{\ell})$ by Lemma~\ref{lemma_identities-N6}, Lemma~\ref{lemma-N6} implies that the algebra $\CCX/ \Id(\mathbf{N}_6^{\ell}) $ is the $\CC$-span of the set $\ov{\B}$. By part (b) of Remark~\ref{remark_key} and Lemma~\ref{lemma_N6_min}, to complete the proof of the theorem it is enough to show that $\ov{\B}$ is a basis for $\CCX/ \Id(\mathbf{N}_6^{\ell})$.  By part (a) of Remark~\ref{remark_key} it is enough to show that for every multidegree $\un{\de}=(\de_1,\ldots,\de_n)$ with $\de_1,\ldots,\de_n>0$ and $m=|\un{\de}|\geqslant 1$ if 
$$f(x_1,\ldots,x_n) = \sum_{i=1}^n \al_i\, \big(x_{i}x_{1}^{\de_1} \cdots \cdot x_{i-1}^{\de_{i-1}} x_i^{\de_i-1} x_{i+1}^{\de_{i+1}} \cdots x_n^{\de_n}\big)_{L} + 
\be\,w \in \Id(\mathbf{N}_6^{\ell}),$$
where $\al_1,\ldots,\al_n,\be\in\CC$ and 
\begin{enumerate}
\item[$\bullet$] $w=(x_{j_1} (x_{j_2} x_{j_3})) x_{j_4} \cdots x_{j_m}$ for 
$(j_1,\ldots,j_m)=(\underbrace{1,\ldots,1}_{\de_1},\underbrace{2,\ldots,2}_{\de_2},\ldots,\underbrace{n\ldots,n}_{\de_n})$, in case  $m\geqslant 3$;

\item[$\bullet$] $w=0$, $\be=0$, in case $m=1,2$,
\end{enumerate}
then $\al_1=\cdots=\al_n=\be=0$. In case $m=1,2$ the polynomial $f$ belongs to the list: $\al_1 x_1$, $\al_1 x_1^2$, $\al_1 x_1 x_2 + \al_2 x_2 x_1$, and the required statement follows easily. Therefore, we assume that $m\geqslant 3$.

\medskip
\noindent{\bf (a)}
At first, assume that $\un{\de}=1^n$ for $n\geqslant3$. Then
$$f(x_1,\ldots,x_n) = \sum_{i=1}^n \al_i\, v_{i,n}(x_1,\ldots,x_n) + 
\be\, w_n(x_1,\ldots,x_n),$$
where $v_{i,n}$ and $w_n$ were defined in~\eqref{eq_notations}. Denote 
$$\un{a}_j=(e_2,\ldots,e_2,\underbrace{e_1}_{j^{\rm th}\,\text{position}},e_2,\ldots,e_2).$$
Note that for every $1\leqslant i,j\leqslant n$ we have
\begin{eq}\label{eq_vi}
v_{i,n}(\un{a}_j) =
\left\{
\begin{array}{rl}
e_1 &, \text{ if } i=j\\
\ell e_1 &, \text{ if } i\neq j\\
\end{array}
\right.,
\end{eq}
\begin{eq} \label{eq_w}
w_n(\un{a}_j) =
\left\{
\begin{array}{rl}
e_1 &, \text{ if } j=1\\
\ell e_1 &, \text{ if } j\not\in\{1,3\}\\
\ell^2 e_1 &, \text{ if } j=3\\
\end{array}
\right..
\end{eq}%
Since $f(e_2,\ldots,e_2)=0$, then $\be=-\al_1-\cdots - \al_n$. 
Given $j\not\in \{1,3\}$, equality $f(\un{a}_j)=0$ implies that 
$$\sum_{1\leqslant i\leqslant n,\; i\neq j} \ell \al_i + \al_j + \ell \be =0.$$
Thus, $\al_j(1-\ell)=0$, i.e., $\al_j=0$. Since $f(\un{a}_1)=0$, we obtain $0=\al_1 + \ell \al_3 + \be = (\ell - 1)\al_3$, i.e., $\al_3=0$. Since $f(\un{a}_3)=0$, we obtain $0=\ell \al_1+ \be \ell^2=(\ell-\ell^2)\al_1$. Therefore, $\al_1=\cdots=\al_n=\be=0$.

\medskip
\noindent{\bf (b)}
Assume that $\un{\de}$ is an arbitrary with $\de_1,\ldots,\de_n>0$ and $m=|\un{\de}|\geqslant 3$. Applying congruences~\eqref{eq_18}, \eqref{eq_lemma_N6_1}, \eqref{eq_lemma_N6_2}, \eqref{eq_lemma_N6_3} to the complete linearization of $f$, we obtain the following multilinear polynomial identity for $\mathbf{N}_6^{\ell}$ of multidegree $1^m$:
$$h(x_1,\ldots,x_m)=\de_1 ! \cdots \de_{n}!\sum_{r=1}^m \frac{1}{\de_{\un{\de}|r|}} \al_{\un{\de}|r|}\, v_{r,m}(x_1,\ldots,x_m) + \de_1 ! \cdots \de_{n}!\, \be\, w_m(x_1,\ldots,x_m) + \be h_0,$$
where $h_0$ is a linear combination of monomials $\{ v_{r,m}(x_1,\ldots,x_m)\,|\, 1\leqslant r\leqslant m\}$. By part (a) of the proof, we obtain that $\be=0$. Therefore, part (a) of the proof implies $\al_1=\cdots=\al_n=0$. The proof of the theorem is concluded.
\end{proof_of}

Part 2 of Theorem~\ref{theorem-N6} implies the following corollary.

%----------------------------------------------------------------
\begin{cor}\label{cor_N6}
Let $\B'$ be the following subset of $\CCX:$
\begin{enumerate}
    \item[(a)]  $x_{j_1} x_{j_2} \cdots x_{j_n} \quad\quad\quad\quad\quad\, (n\geqslant 1,\;\;  1\leqslant j_{1}\leqslant j_2 \leqslant \cdots \leqslant j_{n})$,
    \item[(b)] $(x_{j_1}, x_{j_2}, x_{j_3}) x_{j_4} \cdots x_{j_n} \quad (n\geqslant 3,\;\;1 \leqslant j_{1}\leqslant j_2 \leqslant \cdots \leqslant j_{n})$,
    \item[(c)] $[x_{j_1},x_{i}]x_{j_2}\cdots x_{j_{n-1}}  \quad\quad\;\; (n\geqslant 2,\;\; 1\leqslant j_{1}\leqslant j_2 \leqslant \cdots \leqslant j_{n-1},\;\; i> j_1)$. 
    \end{enumerate}
    Then the image of $\B'$ in the relatively free algebra
$\CCX/\Id(\mathbf{N}^{\ell}_6)$ is its basis.
\end{cor}

%=================================================================
%=================================================================
\section{Nonassociative case: algebra $\mathbf{N}_5$}\label{section-nonassociative-N5}

Recall that the algebra $\mathbf{N}_5$
has the following multiplication table:
\[
\mathbf{N}_5:  e_1 e_1=0,\qquad
e_1 e_2=e_1, \qquad
e_2 e_1=0, \qquad
e_2 e_2=e_1+e_2.
\]
The algebra $\mathbf{N}_5$ is not associative, since $(e_2,e_2,e_2)=e_1\neq0$. In this section we will prove the following result.

%-----------------------------------------------------------------
\begin{theorem}\label{theorem-N5}
\noindent{1.}    The ideal $\Id(\mathbf{N}_5)$ is minimally generated, as a T-ideal, by the polynomials \eqref{eq_15}, \eqref{eq_16} and $x_1[x_2,x_3]$.

\smallskip
\noindent{2.} 
Let $\B$ be a set from Theorem~\ref{theorem-N6}. Then the image $\ov{\B}$ of $\B$ in the relatively free algebra
$\CCX/\Id(\mathbf{N}_5)$ is its basis.
\end{theorem}

As in Section~\ref{section-nonassociative-N6}, define elements of $\B$ of {\it type} (a) or (b), and the subspaces $V,W$ of $\CCX$. 
Let $\J$ be the T-ideal of $\CCX$ generated by polynomials~\eqref{eq_15}, \eqref{eq_16}, $x_1[x_2,x_3]$. 
Introduce the following equivalence relation  $\equiv$ on $\CCX$, which is called the {\it congruence} modulo $\J$: for $f,g\in \CCX$ denote
$$
f\equiv g\quad\mathrm{if~and~only~if}\quad f+\J=g+\J.
$$%
Since polynomials~\eqref{eq_15} and $x_1[x_2,x_3]$ lie in $\J$, we obtain
\begin{equation}\label{eq_33}
(x_1 x_3) x_2 \equiv (x_1 x_2) x_3,
\end{equation}%
\begin{equation}\label{eq_1[23]}
x_1 (x_3 x_2) \equiv x_1 (x_2 x_3),
\end{equation}%
respectively. The following remark is straightforward.

%-----------------------------------------------------------------
\begin{remark}\label{remark_identities-N5}
The algebra $\mathbf{N}_5$ satisfies the polynomial identity $x_1[x_2,x_3]$, i.e.,  $\J\subseteq  \Id(\mathbf{N}_5)$.
\end{remark}

%----------------------------------------------------------------
\begin{lemma}\label{lemma_N5_min}
The generating set~\eqref{eq_15}, \eqref{eq_16}, $x_1[x_2,x_3]$  for the T-ideal $\J$ is minimal.
\end{lemma}
\begin{proof} Let $L_1,L_2\subset \CCX$ be the sets defined in the proof of Lemma~\ref{lemma_N6_min}. Denote by $L_3$ the $\CC$-span of $x_{\si(1)}[x_{\si(2)},x_{\si(3)}]$, $\si\in\Sym_3$.  Then the claim of the lemma follows from the following three claims.

\medskip
\noindent{\bf 1.} {\it The polynomial~\eqref{eq_15} does not belong to $L_2+L_3$.}

Assume the contrary. Then in the relatively free algebra $\CCX / \Id( x_1(x_2 x_3) )$ we have 
$$(x_1x_2)x_3 - (x_1 x_3) x_2 = \al_1 [x_2,x_3] x_1 + \al_2 [x_1,x_3] x_2 + \al_3 [x_1,x_2]x_3$$
for some $\al_1,\al_2,\al_3\in\CC$. Considering the coefficients of $(x_1x_2)x_3$ and $(x_2x_1)x_3$, respectively, we obtain that $1=\al_3$ and $0=\al_3$, respectively; a contradiction.

\medskip
\noindent{\bf 2.} {\it The polynomial~\eqref{eq_16} does not belong to $L_1+L_3$.}

Assume the contrary. Then in the relatively free algebra $\CCX / \Id( x_1(x_2 x_3) )$ we have 
$$(x_1x_2)x_3 - (x_2 x_1) x_3 = \al_1 ((x_1 x_2) x_3 - (x_1 x_3) x_2 ) + \al_2 ((x_2 x_1) x_3 - (x_2 x_3)x_1 ) + \al_3 ((x_3 x_1) x_2 - (x_3 x_2) x_1)$$
for some $\al_1,\al_2,\al_3\in\CC$. Considering the coefficients of $(x_1x_2)x_3$ and $(x_1x_3)x_2$, respectively, we obtain that $1=\al_1$ and $0=\al_1$, respectively; a contradiction.

\medskip
\noindent{\bf 3.} {\it The polynomial $x_1[x_2,x_3]$ does not belong to $L_1+L_2$.}

Assume the contrary. Then in the relatively free algebra $\CCX / \Id( (x_1x_2) x_3 )$ we have 
$$x_1 [x_2, x_3] = \al_1 ( x_2 (x_3 x_1) - x_3 (x_2 x_1) ) + \al_2 (x_1 (x_3 x_2) - x_3 (x_1 x_2)) + \al_3 (x_1 (x_2 x_3) - x_2 (x_1 x_3))$$
for some $\al_1,\al_2,\al_3\in\CC$. Considering the coefficients of $x_1(x_2x_3)$ and $x_2(x_1x_3)$, respectively, we obtain that $1=\al_3$ and $0=\al_3$, respectively; a contradiction.

\end{proof}

%-----------------------------------------------------------------
\begin{lemma}\label{lemma_identities_J}
The following polynomials lay in the T-ideal $\J$: \eqref{eq_19}, \eqref{eq_20}, \eqref{eq_21}. Moreover, 
\begin{equation}\label{eq_28}
x_1((x_2 x_3) x_4)\equiv  (x_1(x_2x_3))x_4 - ((x_2x_1)x_4)x_3 + (x_2(x_1x_4))x_3,
\end{equation}
\begin{eq}\label{eq_29new}
   %x_1(x_2(x_3x_4)) - 2 (x_1(x_2 x_3))x_4 + ((x_1x_2)x_3)x_4.
x_1(x_2(x_3x_4)) \equiv (x_1(x_3x_2))x_4 - ((x_3 x_1)x_4)x_2 + (x_3(x_1x_4))x_2. 
\end{eq}
\end{lemma}
\begin{proof}
By congruence~\eqref{eq_33}, $(x_1,x_2,x_3)-(x_1,x_3,x_2)\equiv  - x_1[x_2,x_3]$. Hence, the polynomial $x_1[x_2,x_3]$ implies that
\begin{eq}\label{eq_starN5}
(x_1,x_2,x_3)\equiv (x_1,x_3,x_2).
\end{eq}%

\noindent{}Consequently applying polynomial~\eqref{eq_16}, congruence~\eqref{eq_starN5} and $x_1[x_2,x_3]\equiv0$, we obtain $([x_1,x_2],x_3,x_4)\equiv (x_3,[x_1,x_2],x_4)\equiv (x_3,x_4,[x_1,x_2]) = (x_3x_4)[x_1,x_2] - x_3(x_4[x_1,x_2]) \equiv 0$, i.e., polynomial~\eqref{eq_19} lies in $\J$.

In the proof of Lemma~\ref{lemma_identities_I} it was shown that polynomial \eqref{eq_20} lies in each T-ideal containing polynomial~\eqref{eq_15} as well as that  polynomial \eqref{eq_21} lies in each T-ideal containing polynomials \eqref{eq_19}, \eqref{eq_20}. 

Since polynomials~\eqref{eq_16} and~\eqref{eq_20} lay in $\J$, we prove congruence~\eqref{eq_28} exactly in the same way as congruence~\eqref{eq_lemma_N6_4}. 

To prove congruence~\eqref{eq_29new}, we consequently apply congruences~\eqref{eq_1[23]}, \eqref{eq_33}, \eqref{eq_28}: 
$$x_1(x_2(x_3x_4))\equiv x_1((x_3x_4)x_2) \equiv x_1((x_3x_2)x_4) \equiv 
(x_1(x_3x_2))x_4 - ((x_3x_1)x_4)x_2 + (x_3(x_1x_4))x_2.$$
\end{proof}

As in Section~\ref{section-nonassociative-N6}, applying congruence~\eqref{eq_33} we obtain that $V +\J \subset W +\J$ in $\CCX/ \J$. Moreover,  polynomials~\eqref{eq_16}, \eqref{eq_21}, respectively, from $\J$ imply that 
\begin{equation}\label{eq_lemma_N6_1_N5}
x_2(x_1x_3) \equiv x_1(x_2 x_3) + f_1,
\end{equation}
\begin{equation}\label{eq_lemma_N6_3_N5}
(x_4(x_2x_3))x_1 \equiv  (x_1(x_2 x_3)) x_4 + f_3,
\end{equation}
respectively, for some multihomogeneous $f_1,f_3\in V$ with $\mdeg(f_1)=1^3$ and $\mdeg(f_3)=1^4$.

%-------------------------------------------------------------------
\medskip
\begin{proof_of}{of Theorem~\ref{theorem-N5}.}
The proof is an analogue of the proof of Theorem~\ref{theorem-N6}, where we use congruences modulo $\J$ instead of congruences modulo $\I$. Namely, instead of congruences 
\eqref{eq_18}, \eqref{eq_lemma_N6_1}, \eqref{eq_lemma_N6_2}, \eqref{eq_lemma_N6_3}, \eqref{eq_lemma_N6_4}, \eqref{eq_lemma_N6_5} modulo $\I$, respectively, we use congruences  
\eqref{eq_33}, \eqref{eq_lemma_N6_1_N5}, \eqref{eq_1[23]}, 
\eqref{eq_lemma_N6_3_N5}, \eqref{eq_28}, \eqref{eq_29new} modulo $\J$, respectively. New congruence~\eqref{eq_29new} is not the same as the corresponding old congruence \eqref{eq_lemma_N6_5}, but the rest of new congruences coincide with the corresponding old ones. This fact does not change the proof. Namely, we obtain that analogues of Lemmas~\ref{lemma_N_claims},~\ref{lemma-N6} hold, and then we repeat the proof of Theorem~\ref{theorem-N6}. The only difference starts with formula~\eqref{eq_vi}. We should substitute formulas~\eqref{eq_vi}, \eqref{eq_w}, respectively, with
\begin{eq}\label{eq_vi_new}
v_{i,n}(\un{a}_j) =
\left\{
\begin{array}{rl}
e_1 &, \text{ if } i=j\\
0 &, \text{ if } i\neq j\\
\end{array}
\right.,
\end{eq}
\begin{eq} \label{eq_w_new}
w_n(\un{a}_j) =
\left\{
\begin{array}{rl}
e_1 &, \text{ if } j=1\\
0 &, \text{ if } j\neq 1\\
\end{array}
\right.,
\end{eq}%
respectively. Since 
\begin{eq}\label{eq_43}
v_{i,n}(e_2,\ldots,e_2)=(n-1)e_1 + e_2 \;\text{ and }\; w_n(e_2,\ldots,e_2)=(n-2)e_1 + e_2 
\end{eq}
we obtain
$$0=f(e_2,\ldots,e_2)=(\al_1+\cdots+\al_n)((n-1)e_1 + e_2) + \be ((n-2)e_1 + e_2)$$
and, therefore, $\be=0$.  Thus, $f(\un{a}_j)=\al_j e_1 =0$ implies that $\al_j=0$ for all $1\leqslant j\leqslant n$. The rest of the proof is the same. 
\end{proof_of}

Part 2 of Theorem~\ref{theorem-N5} implies the following corollary.

%----------------------------------------------------------------
\begin{cor}\label{cor_N5}
Let $\B'$ be the same subset of $\CCX$ as in Corollary~\ref{cor_N6}. 
    Then the image of $\B'$ in the relatively free algebra
$\CCX/\Id(\mathbf{N}_5)$ is its basis.
\end{cor}

%=================================================================
%=================================================================
\section{Corollaries}\label{section_cor}

%=================================================================
\subsection{Isomorphism problem}

%-----------------------------------------------------------------
\begin{cor}\label{cor_isomorphism}
Any two-dimensional Novikov algebras, which are not associative, are isomorphic if and only if they satisfy the same polynomial identities.
\end{cor}
\begin{proof} There are the following Novikov algebras, which are not associative: $\mathbf{T}_3$, $\mathbf{N}_5$, $\mathbf{N}_6^{\ell}$, where $\ell\in\CC\backslash\{0,1\}$.

The polynomial identity $(x_1x_2)x_3$ for $\mathbf{T}_3$ (see Theorem~\ref{theorem-T3}) is not a polynomial identity for $\mathbf{N}_5$ and $\mathbf{N}_6^{\ell}$, since $e_1(e_2e_2)=e_1\neq0$ in $\mathbf{N}_5$ as well as in $\mathbf{N}_6^{\ell}$. 

The polynomial identity $x_1[x_2,x_3]$ for $\mathbf{N}_5$ (see Remark~\ref{remark_identities-N5}) is not a polynomial identity for $\mathbf{T}_3$ and $\mathbf{N}_6^{\ell}$, since  $e_2[e_1,e_2]=-e_1\neq0$ in $\mathbf{T}_3$ and $e_2[e_1,e_2]=(\ell-\ell^2)e_1\neq0$ in $\mathbf{N}_6^{\ell}$. 

The polynomial identity $f(x_1,x_2,x_3)=(x_1,x_2,x_3)-(x_1,x_3,x_2)-\ell[x_3,x_2]x_1$ for $\mathbf{N}_6^{\ell}$ (see Lemma~\ref{lemma_identities-N6}) is not a polynomial identity for $\mathbf{T}_3$ and $\mathbf{N}_5$, since $f(e_2,e_2,e_1)=-e_1\neq0$ in $\mathbf{T}_3$ and $f(e_2,e_2,e_1)=-\ell e_1\neq0$ in $\mathbf{N}_5$. Finally, $f$ is not a polynomial identity for $\mathbf{N}_6^{\ell'}$ for $\ell'\in\CC\backslash\{0,1,\ell\}$,  since otherwise $[x_3,x_2]x_1\in\Id(\mathbf{N}_6^{\ell'})$, but $[e_1,e_2]e_2=(1-\ell')e_1\neq0$ in $\mathbf{N}_6^{\ell'}$; a contradiction. The proof is concluded.
\end{proof}

%=================================================================
\subsection{Codimensions}\label{section_codim}

Given an algebra $\A$, we define the \textit{codimension sequence} $\{c_n(\A)\}_{n\geq1}$ by
\[
c_n(\A)=\dim_{\CC}P_n(\A),\:\text{where}\:P_n(\A)=\frac{P_n}{P_n\cap \Id(\A)}
\]
and $P_{n}$ is the vector subspace of 
$\CCX $ 
of multilinear polynomials in the variables $x_1$,
$x_{2},\ldots,x_{n}.$

As a corollary of Propositions~\ref{prop-T2}, \ref{prop-N1-N2-N3}, \ref{prop-N4}, we obtain the following result.

%-----------------------------------------------------------------
\begin{cor}\label{cor-codimension-associative}
    The codimensions of
    two-dimensional associative Novikov algebras with non-zero multiplication are
    \begin{itemize}
        \item [$\bullet$]
$c_1(\mathbf{T}_2)=c_2(\mathbf{T}_2)=1$  and  $c_n(\mathbf{T}_2)=0$ for $n\geqslant 3;$
\smallskip
\item[$\bullet$] 
$c_n(\A)=1$   for $n\geqslant 1$,  where 
$\A\in \{\mathbf{N}_1, \mathbf{N}_2,\mathbf{N}_3\};$
\smallskip
\item[$\bullet$]
$c_n(\mathbf{N}_4)=n$  for $n\geqslant 1.$
    \end{itemize}
\end{cor}

As a corollary of Theorems~\ref{theorem-T3}, \ref{theorem-N6}, \ref{theorem-N5}  we obtain the following result.

%-----------------------------------------------------------------
\begin{cor}\label{cor-codimension-nonassociative}
    The codimensions of  
    two-dimensional Novikov algebras, which are not associative, are
\begin{itemize}
    \item [$\bullet$]
    $c_n(\mathbf{T}_3)=n\, \text{ for } n\geqslant 1;$
    \smallskip
    \item[$\bullet$]
    $c_n(\mathbf{N}_5)=n$ for $n=1,2$ and
    $c_n(\mathbf{N}_5)=n+1\, \text{ for } n\geqslant 3;$
    \smallskip
    \item[$\bullet$]
    $c_n(\mathbf{N}^{\ell}_6)=n$ for $n=1,2$ and
    $c_n(\mathbf{N}^{\ell}_6)=n+1\, \text{ for } n\geqslant 3.$
\end{itemize}
\end{cor}

%=================================================================
\section{Polynomial images}\label{section_image}

Assume that  $\A$ is an algebra and $f=f(x_1,\ldots,x_n)\in\CCX$ is a polynomial. Then the image of $\A$ with respect to $f$ is 
$$f(\A)=\{f(a_1,\ldots,a_n)\,|\, a_1,\ldots,a_n\in\A\}.$$ 
We consider the image of a two-dimensional Novikov algebra $\A$ with respect to a multilinear polynomial $f$ of multidegree $1^n$. Assume that $\B$ is a subset of $\CCX$ from Propositions~\ref{prop-T2}, \ref{prop-N1-N2-N3}, \ref{prop-N4} or Theorems~\ref{theorem-T3}, \ref{theorem-N6}, \ref{theorem-N5}, i.e., the image of $\B$ in the relatively free algebra  $\CCX/\Id(\A)$ is its basis. Without loss of generality, we can assume that $f$ lies in the $\CC$-span of $\B$, since there exists a multilinear polynomial $h$ from $\CC$-span of $\B$ with $f-h\in\Id(\A)$; in particular, $f(\A)=h(\A)$. We also assume that $f\not\in\Id(\A)$, since otherwise $f(\A)=\{0\}$. Moreover, we assume that $f\neq \al x_1$, where  $\al\in\CC\backslash\{0\}$, since otherwise $f(\A)=\A$. In other word, we have $n>1$.

%-----------------------------------------------------------------
\begin{theorem}\label{theo_image} Consider a two-dimension Novikov algebra $\A$ and $f\in\CCX$ of multidegree $1^n$, which is not a polynomial identity for $\A$, where $n>1$. Assume that $f$ lies in the $\CC$-span of $\B$ (see above). Then
\begin{enumerate}
\item[$\bullet$] $f(\mathbf{T}_2)=\CC e_2$;

\item[$\bullet$] $f(\A)=\CC e_1$ in case $\A\in\{\mathbf{T}_3,\,\mathbf{N}_2\}$;

\item[$\bullet$] $f(\A)=\A$ in case 
$\A\in\{\mathbf{N}_1,\,\mathbf{N}_3\}$;

\item[$\bullet$] $f(\A)=\CC e_1$ in case $s(f)=0$ and $\A\in\{\mathbf{N}_4,\;\mathbf{N}_5,\;\mathbf{N}_6^{\ell}\}$; 

\item[$\bullet$] $f(\A)=\A$ in case $s(f)\neq0$  and $\A\in\{\mathbf{N}_4,\;\mathbf{N}_5\}$; 

\item[$\bullet$] $f(\mathbf{N}_6^{\ell})=\CC e_2$, if $n=2$, $\ell = -1$ and $f= \al v_{1,n} + \al v_{2,n}$ for some non-zero $\al\in\CC$;

\item[$\bullet$] $f(\mathbf{N}_6^{\ell})=\CC e_2$, if $n>2$ and there exists non-zero $\al\in\CC$ such that
$$
\frac{1}{\al}f= \frac{\ell^2 +n\ell - \ell + 1}{\ell^2} v_{1,n} + v_{2,n}+
\frac{2\ell - n \ell -1}{\ell} v_{3,n} + \sum_{i=4}^n v_{i,n} + \frac{\ell - n \ell -1}{\ell^2} w_n;
$$
\item[$\bullet$] $f(\mathbf{N}_6^{\ell})=\mathbf{N}_6^{\ell}$ in case $s(f)\neq0$ and the conditions on $f$ from the previous two items do not hold. 
\end{enumerate}
Here $\ell\in\CC\backslash\{0,1\}$.
\end{theorem}

%-----------------------------------------------------------------
\begin{cor}\label{cor_image}
The multilinear image of a two-dimensional Novikov algebra is a linear space.
\end{cor}

We split the proof of Theorem~\ref{theo_image} into Propositions~\ref{prop_image1}, \ref{prop_image_N4}, \ref{prop_image_N5}, \ref{prop_image_N6} (see below).

%-----------------------------------------------------------------
\begin{prop}\label{prop_image1} Consider $f\in\CCX$ as in Theorem~\ref{theo_image}.  Then 
\begin{enumerate}
\item[$\bullet$] $f(\A)=\CC e_2$ in case $\A=\mathbf{T}_2$;

\item[$\bullet$] $f(\A)=\CC e_1$ in case $\A\in\{\mathbf{T}_3,\,\mathbf{N}_2\}$;

\item[$\bullet$] $f(\A)=\A$ in case 
$\A\in\{\mathbf{N}_1,\,\mathbf{N}_3\}$.
\end{enumerate}
\end{prop}
\begin{proof} Assume  $\A\in\{ \mathbf{T}_2,\,\mathbf{T}_3,\,\mathbf{N}_2 \}$. Considering the multiplication table of $\A$ we can see that $f(\A)\subset \CC e_i$ for some $i=1,2$. Since $f$ is not a polynomial identity for $\A$ and $f$ is multilinear, we obtain that $f(\A)= \CC e_i$. The required statement is proven.

Assume $\A\in\{\mathbf{N}_1,\,\mathbf{N}_3\}$. By Proposition~\ref{prop-N1-N2-N3}, we have $f=\al x_1\cdots x_n$ for a non-zero $\al\in\CC$. For each $1\leqslant i\leqslant n$ consider $b_i=\be e_1 + \be'_i e_2\in\A$ with $\be_i,\be'_i\in\CC$.

If $\A=\mathbf{N}_1$, then $f(b_1,\ldots,b_n)=\al (\be_1\cdots \be_n e_1 + \be'_1\cdots \be'_n e_2)$ and the required statement is proven.

If $\A=\mathbf{N}_3$, then $f(b_1,e_1,\ldots,e_1)=\al (\be_1 e_1 + \be'_1 e_2)$ and the required statement is proven.
\end{proof}

%-----------------------------------------------------------------
\begin{prop}\label{prop_image_N4}  Consider $f\in\CCX$ as in Theorem~\ref{theo_image}. Then for $\A=\mathbf{N}_4$ we have 
\begin{enumerate}
\item[$\bullet$] $f(\A)=\CC e_1$ in case $s(f)=0$; 

\item[$\bullet$] $f(\A)=\A$ in case $s(f)\neq0$. 
\end{enumerate}
\end{prop}
\begin{proof} By Proposition~\ref{prop-N4}, we have $f=\sum_{i=1}^n \al_i v_{i,n}$, where $\al_i\in\CC$ and $v_{i,n}$ were defined in~\eqref{eq_notations}. Since $f$ is non-zero, using polynomial identity $x_1[x_2,x_3]$  we can assume that $\al_1\neq0$. For each $1\leqslant i\leqslant n$ consider $b_i=\be e_1 + \be'_i e_2\in \mathbf{N}_4$ with $\be_i,\be'_i\in\CC$. Since $e_k e_1=0$ for every $k=1,2$ and $b e_2 =b$ for every $b\in \mathbf{N}_4$, we have  that $v_{1,n}(b_1,\ldots,b_n)=b_1 \cdots b_n$ is equal to  
$$\be_1 \be'_2 \cdots \be'_n e_1e_2\cdots e_2 + \be'_1 \cdots \be'_n e_2\cdots e_2 = \be_1 \be'_2 \cdots \be'_n e_1 + \be'_1 \cdots \be'_n e_2.$$
\noindent{}Therefore, 
$$
\begin{array}{rcl}
f(b_1,\ldots,b_n) & = & \sum\limits_{i=1}^n \al_i \left(\be'_1 \cdots \be'_{i-1} \be_i \be'_{i+1}\cdots \be'_n e_1 + \be'_1\cdots \be'_n e_2 \right) \\
& = & \Big(\sum\limits_{i=1}^n \al_i \be'_1\cdots \be'_{i-1}\be_i\be'_{i+1}\cdots \be'_n\Big)e_1 + \be'_1\cdots \be'_n s(f)e_2.\\
\end{array}
$$
Hence, in case $s(f)=0$ we have $f(\mathbf{N}_4)=\CC e_1$. In case $s(f)\neq0$ we have $f(b_1,e_2,\ldots,e_2)=\al_1\be_1 e_1 + \be'_1 s(f) e_2$; thus, $f(\mathbf{N}_4)=\mathbf{N}_4$.
\end{proof}

%-----------------------------------------------------------------
\begin{prop}\label{prop_image_N5}  Consider $f\in\CCX$ as in Theorem~\ref{theo_image}. Then for $\A=\mathbf{N}_5$ we have
\begin{enumerate}
\item[$\bullet$] $f(\A)=\CC e_1$ in case $s(f)=0$; 

\item[$\bullet$] $f(\A)=\A$ in case $s(f)\neq0$. 
\end{enumerate}
\end{prop}
\begin{proof} By Theorem~\ref{theorem-N5}, we have $f=\sum_{i=1}^n \al_i v_{i,n} + \al_0 w_n$, where $\al_i\in\CC$. 
%Since $f$ is non-zero, using polynomial identity $x_1[x_2,x_3]$  we can assume that $\al_1\neq0$. 
For each $1\leqslant i\leqslant n$ consider $b_i=\be e_1 + \be'_i e_2\in \mathbf{N}_5$ with $\be_i,\be'_i\in\CC$. Note that $e_ke_1=0$ in $\mathbf{N}_5$ for every $k=1,2$.

Assume $n=2$. Then $f=\al_1 x_1 x_2 + \al_2 x_2 x_1$. Permuting $x_1$ and $x_2$, without loss of generality, we can assume that $\al_1\neq0$. Then 
$$
\begin{array}{rcl}
f(b_1,b_2) & = &\al_1 (\be_1 \be'_2 e_1e_2 + \be'_1\be'_2 e_2 e_2) + 
\al_2 (\be'_1 \be_2 e_1e_2 + \be'_1\be'_2 e_2 e_2) \\
&=& (\al_1 \be_1 \be'_2 + \al_2\be'_1\be_2 + (\al_1+\al_2)\be'_1\be'_2)e_1  +  (\al_1+\al_2)\be'_1\be'_2 e_2.\\
\end{array}
$$
Hence, $f(\mathbf{N}_5)\subset \CC e_1$ in case $\al_1+\al_2=0$. On the other hand, $f(b_1,e_2)=(\al_1 \be_1 + (\al_1+\al_2)\be'_1)e_1  +  (\al_1+\al_2)\be'_1 e_2$ concludes the proof in case $n=2$.

Assume $n\geqslant 3$. Applying equalities~\eqref{eq_43} we can see that $v_{1,n}(b_1,\ldots,b_n)=b_1 \cdots b_n$ is equal to  
$$\be_1 \be'_2 \cdots \be'_n e_1e_2\cdots e_2 + \be'_1 \cdots \be'_n e_2\cdots e_2 = 
\be_1 \be'_2 \cdots \be'_n e_1 + \be'_1 \cdots \be'_n ((n-1)e_1 + e_2)$$
and $w_n(b_1,\ldots,b_n)=(b_1(b_2b_3))b_4\cdots b_n$  is equal to  
\begin{multline*}
\be_1\be'_2\cdots \be'_n(e_1(e_2e_2))e_2\cdots e_2 + 
\be'_1\be'_2\cdots \be'_n(e_2(e_2e_2))e_2\cdots e_2 = \\
\be_1\be'_2\cdots \be'_n e_1+ 
\be'_1\be'_2\cdots \be'_n((n-2)e_1+e_2).
\end{multline*}
Therefore, $f(b_1,\ldots,b_n)$ is equal to  
$$
\Big(\sum\limits_{i=1}^n \al_i \be'_1\cdots \be'_{i-1} \be_i \be'_{i+1}\cdots \be'_n +  \al_0\be_1\be'_2\cdots \be'_n +\al \be'_1\cdots \be'_n \Big)e_1
+ \be'_1\cdots\be'_n s(f) e_2,
$$
where $\al=(n-1)\sum_{i=1}^n \al_i + (n-2)\al_0$. Thus, in case $s(f)=0$ we have $f(\mathbf{N}_5)\subset \CC e_1$, and, hence, $f(\mathbf{N}_5)=\CC e_1$. 

Assume $s(f)\neq0$. If there exist $2\leqslant i\leqslant n$ such that $\al_i\neq0$, then for $b_j=e_2$ for all $1\leqslant j\leqslant n$ with $j\neq i$ we consider
$$f(b_1,\ldots,b_n)= (\al_i \be_i + \al \be'_i)e_1 + \be'_i s(f)e_2$$
to complete the proof. If $\al_2=\cdots=\al_n=0$, then we consider
$$f(b_1,e_2,\ldots,e_2)= (\al_1 \be_1 + \al_0 \be_1 + \al \be'_1)e_1 + \be'_1 s(f)e_2 = 
(s(f)\be_1 + \al \be'_1)e_1 + \be'_1 s(f)e_2$$
to complete the proof.
\end{proof}

%-----------------------------------------------------------------
\begin{prop}\label{prop_image_N6}  Consider $f\in\CCX$ as in Theorem~\ref{theo_image}. Then for $\A=\mathbf{N}_6$ we have 
\begin{enumerate}
\item[(a)] $f(\A)=\CC e_1$, if $s(f)=0$;

\item[(b)] $f(\A)=\CC e_2$, if $n=2$, $\ell = -1$ and $f= \al v_{1,n} + \al v_{2,n}$ for some non-zero $\al\in\CC$;

\item[(c)] $f(\A)=\CC e_2$, if $n>2$ and there exists non-zero $\al\in\CC$ such that
\begin{eq}\label{eq_image_N6_c}
\frac{1}{\al}f= \frac{\ell^2 +n\ell - \ell + 1}{\ell^2} v_{1,n} + v_{2,n}+
\frac{2\ell - n \ell -1}{\ell} v_{3,n} + \sum_{i=4}^n v_{i,n} + \frac{\ell - n \ell -1}{\ell^2} w_n, 
\end{eq}
\item[(d)] $f(\A)=\A$, otherwise. 
\end{enumerate}
\end{prop}
\begin{proof} By Theorem~\ref{theorem-N6}, we have $f=\sum_{i=1}^n \al_i v_{i,n} + \al_0 w_n$, where $\al_i\in\CC$. 
For each $1\leqslant i\leqslant n$ consider $b_i=\be e_1 + \be'_i e_2\in \mathbf{N}_6$ with $\be_i,\be'_i\in\CC$. 

\medskip
\noindent{\bf 1. }
Assume $n=2$. Then $f=\al_1 x_1 x_2 + \al_2 x_2 x_1$ and
$$
\begin{array}{rcl}
f(b_1,b_2) & = &\al_1 (\be_1 \be'_2 e_1e_2 + \be'_1\be_2 e_2 e_1 + \be'_1\be'_2 e_2 e_2) + 
\al_2 (\be'_1 \be_2 e_1e_2 +\be_1 \be'_2 e_2e_1 + \be'_1\be'_2 e_2 e_2) \\
&=& ((\al_1 + \ell \al_2)\be_1 \be'_2 + (\ell\al_1+\al_2)\be'_1\be_2)e_1  +  (\al_1+\al_2)\be'_1\be'_2 e_2.\\
\end{array}
$$%
Hence, if $\al_1+\al_2=0$, then case (a) holds. 

Assume, $\al_1+\al_2\neq0$. We have $\al_1 + \ell \al_2= \ell \al_1 + \al_2=0$ if and only if $\al_1=\al_2$ and $\ell=-1$; in this case we obviously have that $f(\mathbf{N}_6^{\ell})=\CC e_2$. Case (b) is proven.

If $\al_1+\ell \al_2\neq0$, then we consider $f(b_1,e_2)=(\al_1 +\ell\al_2) \be_1 e_1 +  (\al_1+\al_2)\be'_1 e_2$ to obtain that case (d) holds.

If $\ell\al_1+\al_2\neq0$, then we consider $f(e_2,b_2)=(\ell\al_1 +\al_2) \be_2 e_1 +  (\al_1+\al_2)\be'_2 e_2$ to obtain that case (d) holds.

\medskip
\noindent{\bf 2. } Assume $n> 2$. Note that for $1\leqslant i_1,\ldots,i_n\leqslant n$ we have
\begin{eq}\label{eq_eee}
e_{i_1}\cdots e_{i_n} = \left\{
\begin{array}{rl}
e_1&, \text{ if } (i_1,\ldots,i_n)=(1,2,\ldots,2)\\
\ell e_1 &, \text{ if } (i_1,\ldots,i_n)=(\underbrace{2,\ldots,2}_{ \geqslant 1\; {\rm times}},1,2,\ldots,2) \\
e_2 &, \text{ if } i_1=\cdots=i_n=2 \\
0&, \text{ otherwise} \\
\end{array}
\right..
\end{eq}
(cf.~formula~\eqref{eq_vi}). Hence, $v_{1,n}(b_1,\ldots,b_n)=b_1 \cdots b_n$ is equal to  
\begin{multline*}
\be_1 \be'_2\cdots \be'_n\, e_1e_2\cdots e_2 + 
\sum_{j=2}^n \be'_1\cdots \be'_{j-1}\be_j\be'_{j+1}\cdots \be'_n \underbrace{e_2 \cdots e_2}_{(j-1) \; {\rm times}} e_1 e_2 \cdots e_2 +
 \be'_1\cdots \be'_n \,e_2\cdots e_2 =  \\
\Big(\be_1 \be'_2\cdots \be'_n  + 
\ell \sum_{j=2}^n \be'_1\cdots \be'_{j-1}\be_j\be'_{j+1}\cdots \be'_n \Big) e_1 +
 \be'_1\cdots \be'_n \,e_2.
\end{multline*}%
Since $b_2b_3=(\be_2\be'_3 + \ell \be'_2 \be_3)e_1 + \be'_2\be'_3e_2$, we obtain
$$
b_1(b_2b_3)=\be e_1 + \be'_1\be'_2\be'_3 e_2 \;\text{ for }\;
\be=\be_1 \be'_2 \be'_3 + \ell \be'_1 \be_2 \be'_3 + \ell^2 \be'_1 \be'_2 \be_3.
$$%
Thus, $w_n(b_1,\ldots,b_n)=(b_1(b_2b_3))b_4\cdots b_n$  is equal to  
\begin{multline*}
\be \be'_4\cdots \be'_n\, e_1e_2\cdots e_2 + 
\be'_1\be'_2\be'_3 \sum_{j=4}^n \be'_4\cdots \be'_{j-1}\be_j\be'_{j+1}\cdots \be'_n \underbrace{e_2 \cdots e_2}_{ (j-3) \; {\rm times}} e_1 e_2 \cdots e_2   \\
+\be'_1\cdots \be'_n \,e_2\cdots e_2 =
 \Big(\be \be'_4\cdots \be'_n  + 
\ell \sum_{j=4}^n \be'_1\cdots \be'_{j-1}\be_j\be'_{j+1}\cdots \be'_n \Big) e_1 +
 \be'_1\cdots \be'_n \,e_2.
\end{multline*}%
Therefore, $f(b_1,\ldots,b_n)$ is equal to  
\begin{multline*}
\Big( \sum_{i=1}^n \al_i \be'_1\cdots\be'_{i-1}\be_i\be'_{i+1}\cdots \be'_n
+ \ell \sum_{i=1}^n\al_i \sum_{1\leqslant j\leqslant n,\; j\neq i} 
\be'_1\cdots\be'_{j-1}\be_j\be'_{j+1}\cdots \be'_n  \Big) e_1 \\
+ \be'_1\cdots \be'_n \Big(\sum_{i=1}^n \al_i \Big) e_2 
+ \al_0 \Big( \be \be'_4\cdots \be'_n + 
\ell \sum_{j=4}^n \be'_{1}\cdots \be'_{j-1} \be_j \be'_{j+1} \cdots \be'_n\Big) e_1 +
\al_0 \be'_1\cdots \be'_n e_2 = \\
\Bigg( \sum_{i=1}^n \be'_1\cdots\be'_{i-1}\be_i\be'_{i+1}\cdots \be'_n 
\Big( \al_i + \ell \!\!\!\!\!\! \sum_{1\leqslant j\leqslant n,\; j\neq i} \al_j\Big)
+ \al_0 \be \be'_4\cdots \be'_n \\ 
+ \ell \al_0 \sum_{j=4}^n \be'_{1}\cdots \be'_{j-1} \be_j \be'_{j+1} \cdots \be'_n \Bigg)e_1 
 +  \be'_1\cdots \be'_n s(f) \,e_2.\\
\end{multline*}%
Using the definition of $\be$, we obtain that $f(b_1,\ldots,b_n)$ is equal to
$$
\begin{array}{l}
\Bigg( \be_1 \be'_2\be'_3 \cdots \be'_n \Big(\al_0 + \al_1 + \ell \al_2 + \cdots+\ell \al_n\Big) \\
+ \be'_1 \be_2 \be'_3 \cdots \be'_n \Big(\ell \al_0 + \ell \al_1 + \al_2+ \ell \al_3 + \cdots+\ell \al_n\Big) \\
+ \be'_1 \be'_2 \be_3 \be'_4 \cdots \be'_n \Big(\ell^2 \al_0 + \ell \al_1 + \ell \al_2+  \al_3 + \ell \al_4 \cdots+\ell \al_n\Big) \\
+\sum\limits_{i=4}^n \be'_1 \cdots \be'_{i-1} \be_i \be'_{i+1} \cdots \be'_n \Big(\ell \al_0 + \ell \al_1 + \cdots + \ell \al_{i-1}+ \al_i + \ell \al_{i+1} + \cdots+\ell \al_n\Big)  \Bigg)e_1 \\
 + \be'_1\cdots \be'_n\, s(f) \,e_2 = \\
\Big(\sum\limits_{i=1}^n \be'_1 \cdots \be'_{i-1} \be_i \be'_{i+1} \cdots \be'_n \,\ga_i  \Big)e_1 
 +  \be'_1\cdots \be'_n s(f) \,e_2, \text{ where } \\
\end{array}
$$
$$\ga_1 = \ell\, s(f) - (\ell-1)(\al_0 + \al_1),\;\; 
\ga_2=\ell\, s(f) - (\ell-1)\al_2$$
$$\ga_3= \ell\, s(f) + (\ell^2-\ell)\al_0 - (\ell-1)\al_3,\;\; 
\ga_i= \ell\, s(f) - (\ell-1)\al_i \;\text{ for }\;4\leqslant i\leqslant n.$$

\noindent{}Thus, in case $s(f)=0$ we have $f(\mathbf{N}_6^{\ell})\subset \CC e_1$, and, hence, $f(\mathbf{N}_6^{\ell})= \CC e_1$.

We claim that 
\begin{eq}\label{claim_image_N6}
\ga_i=0 \text{ for all }1\leqslant i\leqslant n \;\text{ if and only if formula~\eqref{eq_image_N6_c} holds for some non-zero }\al\in\CC.
\end{eq}

To prove claim~\eqref{claim_image_N6}, we assume that $\ga_i=0$ for all $i$. Then $\al_2=\al_4=\cdots=\al_n=\ell\,s(f)/(\ell-1)$ and $\al_1=\al_2-\al_0$. Hence, $\al_3=\al_2+\ell\al_0$. Since 
$\frac{\ell-1}{\ell} \al_2=s(f)=\al_0+\cdots+\al_n=n \al_2 + \ell \al_0$, we obtain $\al_0 = \frac{\ell - n\ell - 1}{\ell^2}\al_2$ and formula~\eqref{eq_image_N6_c} follows. The inverse statement is straightforward. 

In case $\ga_i=0$ for all $1\leqslant i\leqslant n$, we have $f(\mathbf{N}_6^{\ell})= \CC e_2$. Therefore, using claim~\eqref{claim_image_N6}, we conclude the proof of part (c) of Proposition~\ref{prop_image_N6}.

Assume $\ga_i\neq 0$ for some $1\leqslant i\leqslant n$. For $b_j=e_2$ for all $1\leqslant i\leqslant n$ with $j\neq i$ we consider 
$f(b_1,\ldots,b_n)= \be_i \ga_i\, e_1 + \be'_i s(f) \,e_2$
to obtain that $f(\mathbf{N}_6^{\ell})=\mathbf{N}_6^{\ell}$. Part (d) is proven.
\end{proof}

%-----------------------------------------------------------------
\subsection*{Acknowledgments}%
The results of this paper were obtained under collaboration of the research group of UFRN ``Algebra and Universal algebraic geometry". The third author was supported by CNPq 405779/2023-2.

\bibliography{reference}

\bibliographystyle{abbrvurl}

%=====================================================
\end{document}